\newcommand{\R}{\mathbb R}

\newcommand{\N}{\mathbb N}
\newcommand{\dist}{\mathrm{dist}}
\newcommand{\cat}{\mathrm{cat}}

\newcommand{\Dim}{\mathrm{dim}}

\newcommand{\second}{\mathrm I\!\mathrm I}

\newcommand{\Fcal}{\mathcal F}

\newcommand{\Rcal}{\mathcal R}

\newcommand{\Dcal}{\mathcal D}
\newcommand{\Hcal}{\mathcal H}
\newcommand{\Ucal}{\mathcal U}
\newcommand{\Mcal}{\mathfrak M}
\newcommand{\Xcal}{\mathcal X}
\newcommand{\Ycal}{\mathcal Y}

\newcommand{\Ddt}{\tfrac{\mathrm D}{\mathrm dt}}
\newcommand{\Dds}{\tfrac{\mathrm D}{\mathrm ds}}

\documentclass[twoside,a4paper,10pt]{amsart}

\usepackage{amsmath}
\usepackage{times}
\usepackage{hyperref}
\usepackage{dsfont}
\usepackage{graphicx}
\usepackage{epsfig}
\usepackage{color}

\numberwithin{equation}{section}

\title[Multiple Brake orbits in $m$-disks]%
{Multiple brake orbits in $\mathbf m$--dimensional disks}

\author[R. Giamb\`o]{Roberto Giamb\`o}
\author[F. Giannoni]{Fabio Giannoni}
\address{Dipartimento di Matematica e Informatica,\hfill\break\indent
Universit\`a di Camerino, Italy}
\email{roberto.giambo@unicam.it, fabio.giannoni@unicam.it}
\author[P. Piccione]{Paolo Piccione}
\address{Departamento de Matem\'atica \hfill\break\indent Universidade de S\~ao Paulo \hfill\break\indent
Rua do Mat\~ao, 1010 \hfill\break\indent 05508-090 S\~ao Paulo, SP, Brazil}
\email{piccione.p@gmail.com}

\urladdr{http://www.ime.usp.br/\~{}piccione}
\date{February 11th, 2015}

\begin{document}


\theoremstyle{plain}\newtheorem{teo}{Theorem}[section]
\theoremstyle{plain}\newtheorem{prop}[teo]{Proposition}
\theoremstyle{plain}\newtheorem{lem}[teo]{Lemma}
\theoremstyle{plain}\newtheorem{cor}[teo]{Corollary}
\theoremstyle{definition}\newtheorem{defin}[teo]{Definition}
\theoremstyle{remark}\newtheorem{rem}[teo]{Remark}
\theoremstyle{definition}\newtheorem{example}[teo]{Example}

\theoremstyle{plain}\newtheorem*{teon}{Theorem}


\begin{abstract}
Let $(M,g)$ be a (complete) Riemannian surface, and let $\Omega\subset M$
be an open subset whose closure is homeomorphic to a disk. We
prove that if $\partial\Omega$ is smooth and it satisfies a
strong concavity assumption, then there are at least two distinct
orthogonal geodesics  in $\overline\Omega=\Omega \bigcup\partial\Omega$.
Using the results given in \cite{GGP1}, we then obtain a proof of the
existence  of two distinct {\em brake orbits\/} for a class of Hamiltonian systems.
In our proof we shall use recent deformation results proved in \cite{esistenza}.
\end{abstract}

\maketitle

\renewcommand{\contentsline}[4]{\csname nuova#1\endcsname{#2}{#3}{#4}}
\newcommand{\nuovasection}[3]{\medskip\hbox to \hsize{\vbox{\advance\hsize by -1cm\baselineskip=12pt\parfillskip=0pt\leftskip=3.5cm\noindent\hskip -2cm #1\leaders\hbox{.}\hfil\hfil\par}$\,$#2\hfil}}
\newcommand{\nuovasubsection}[3]{\medskip\hbox to \hsize{\vbox{\advance\hsize by -1cm\baselineskip=12pt\parfillskip=0pt\leftskip=4cm\noindent\hskip -2cm #1\leaders\hbox{.}\hfil\hfil\par}$\,$#2\hfil}}

\tableofcontents

\section{Introduction}
\label{sec:intro}
In this paper we will use a non-smooth version
of the Ljusternik--Schnirelman theory to prove the
existence of multiple orthogonal geodesic chords
in a Riemannian manifolds with boundary. This fact, together
with the results in \cite{GGP1}, gives a multiplicity
result for   brake orbits of a class of
Hamiltonian systems. Let us recall a few basic facts and
notations from \cite{GGP1}.

\begin{subsection}{Geodesics in Riemannian Manifolds with Boundary}
\label{sub:geodes} Let $(M,g)$ be a smooth (i.e., of class $C^2$)
Riemannian manifold with $\Dim(M)=m\ge2$, let $\dist$ denote the
distance function on $M$ induced by $g$; the symbol $\nabla$ will
denote the covariant derivative of the Levi-Civita connection of
$g$, as well as the gradient differential operator for smooth maps
on $M$. The Hessian $\mathrm H^f(q)$ of a smooth map $f:M\to\R$ at
a point $q\in M$ is the symmetric bilinear form $\mathrm
H^f(q)(v,w)=g\big((\nabla_v\nabla f)(q),w\big)$ for all $v,w\in
T_qM$; equivalently, $\mathrm H^f(q)(v,v)=\frac{\mathrm
d^2}{\mathrm ds^2}\big\vert_{s=0} f(\gamma(s))$, where
$\gamma:\left]-\varepsilon,\varepsilon\right[\to M$ is the unique
(affinely parameterized) geodesic in $M$ with $\gamma(0)=q$ and
$\dot\gamma(0)=v$. We will denote by $\Ddt$ the covariant
derivative along a curve, in such a way that $\Ddt\dot\gamma=0$ is the
equation of the geodesics. A basic reference on the background material
for Riemannian geometry is \cite{docarmo}.

Let $\Omega\subset M$ be an open subset;
$\overline\Omega=\Omega\bigcup\partial \Omega$ will denote its
closure. In this paper we will use a somewhat
strong concavity assumption for compact subsets of $M$, that we
will call "strong concavity" below, and which is stable by
$C^2$-small perturbations of the boundary.

If $\partial \Omega$ is a smooth embedded submanifold of $M$, let
$\second_{\mathfrak n}(x):T_x(\partial\Omega)\times
T_x(\partial\Omega)\to\R$ denote the {\em second fundamental form
of $\partial\Omega$ in the normal direction $\mathfrak n\in
T_x(\partial\Omega)^\perp$}. Recall that $\second_{\mathfrak
n}(x)$ is a symmetric bilinear form on $T_x(\partial\Omega)$
defined by:
\[\phantom{\qquad v,w\in T_x(\partial\Omega),}\second_{\mathfrak n}(x)(v,w)=g(\nabla_vW,
\mathfrak n),\qquad v,w\in T_x(\partial\Omega),\] where $W$ is any
local extension of $w$ to a smooth vector field along
$\partial\Omega$.

\begin{rem}\label{thm:remphisecond}
Assume that it is given a \emph{signed distance function} for $\partial\Omega$, i.e.,
a smooth function $\phi:M\to\R$ with the
property that $\Omega=\phi^{-1}\big(\left]-\infty,0\right[\big)$
and $\partial\Omega=\phi^{-1}(0)$, with $\mathrm d\phi\ne0$ on
$\partial\Omega$.\footnote{One can choose $\phi$ such that
$\vert\phi(q)\vert=\dist(q,\partial\Omega)$ for all $q$ in a
(closed) neighborhood of $\partial\Omega$.} The following equality
between the Hessian $\mathrm H^\phi$ and the second fundamental
form\footnote{%
Observe that, with our definition of $\phi$, then $\nabla\phi$ is
a normal vector to $\partial\Omega$ pointing {\em outwards\/} from
$\Omega$.} of $\partial\Omega$ holds:
\begin{equation}\label{eq:seches}
\phantom{\quad x\in\partial\Omega,\ v\in
T_x(\partial\Omega);}\mathrm H^\phi(x)(v,v)=
-\second_{\nabla\phi(x)}(x)(v,v),\quad x\in\partial\Omega,\ v\in
T_x(\partial\Omega);\end{equation} Namely, if
$x\in\partial\Omega$, $v\in T_x(\partial\Omega)$ and $V$ is a
local extension around $x$ of $v$ to a vector field which is
tangent to $\partial\Omega$, then $v\big(g(\nabla\phi,V)\big)=0$
on $\partial\Omega$, and thus:
\[\mathrm H^\phi(x)(v,v)=v\big(g(\nabla\phi,V)\big)-g(\nabla\phi,\nabla_vV)=-\second_{\nabla\phi(x)}(x)(v,v).\]

For convenience, we will fix throughout the paper a function $\phi$ as above. We observe that,
although the second fundamental form is defined intrinsically,
there is no canonical choice for the function $\phi$
describing the boundary of $\Omega$ as above.
\end{rem}

\begin{defin}\label{thm:defstrongconcavity}
We will say that that $\overline\Omega$
is {\em strongly concave\/} if $\second_{\mathfrak n}(x)$ is
negative definite for all $x\in\partial\Omega$ and all inward
pointing normal direction $\mathfrak n$.
\end{defin}

\noindent
Observe that if $\overline\Omega$ is strongly concave, geodesics starting
tangentially to $\partial\Omega$ remain \emph{inside} $\Omega$.
\smallskip

\begin{rem}\label{thm:newremopencondition}
Strong concavity is evidently a {\em $C^2$-open condition}. Then,
by \eqref{eq:seches}, if $\overline\Omega$ is compact, we deduce
the existence of $\delta_0>0$ such that $\mathrm H^\phi(x)(v,v)<0$
for all $x\in\phi^{-1}\big([-\delta_0,\delta_0]\big)$ and for all
$v\in T_xM$, $v\ne0$, such that $g\big(\nabla\phi(x),v\big)=0$.

A simple contradiction argument based on Taylor expansion shows that, under the above condition, it is $\nabla\phi(q)\ne 0$,
for all $q\in\phi^{-1}([-\delta_0,\delta_0])$.
\end{rem}
\begin{rem}\label{thm:remopencondition}
Let $\delta_0$ be as above.
The strong concavity condition gives us the following property of geodesics,
that will be used systematically throughout the paper:
\begin{equation}\label{eq:1.1bis}
\begin{matrix}
\text{for any  geodesic $\gamma:[a,b]\to\overline\Omega$ with $\phi(\gamma(a))=\phi(\gamma(b))=0$}\\
\text{and $\phi(\gamma(s))<0$ for all $s\in \left]a,b\right[$, there exists $\overline s\in \left]a,b\right[$ such that
$\phi\big(\gamma(\overline s)\big)<-\delta_0$.}
\end{matrix}
\end{equation}
Such property is proved easily by looking at the minimum point
of the map $s\mapsto\phi(\gamma(s))$.
\end{rem}

The main objects of our study are geodesics in $M$ having image in
$\overline\Omega$ and with endpoints orthogonal to
$\partial\Omega$, that will be called {\em orthogonal geodesic chords}:

\begin{defin}\label{thm:defOGC}
A geodesic $\gamma:[a,b]\to M$ is called a {\em geodesic chord\/}
in $\overline\Omega$ if
$\gamma\big(\left]a,b\right[\big)\subset\Omega$ and
$\gamma(a),\gamma(b)\in\partial\Omega$; by a {\em weak geodesic
chord\/} we will mean a geodesic $\gamma:[a,b]\to M$ with image in
$\overline\Omega$ and endpoints
$\gamma(a),\gamma(b)\in\partial\Omega$ and such that $\gamma(s_{0})\in \partial\Omega$ for some $s_{0} \in ]a,b[$. A (weak) geodesic chord is
called {\em orthogonal\/} if $\dot\gamma(a^+)\in
(T_{\gamma(a)}\partial\Omega)^\perp$ and $\dot\gamma(b^-)\in
(T_{\gamma(b)}\partial\Omega)^\perp$, where
$\dot\gamma(\,\cdot\,^\pm)$ denote the one-sided derivatives.

\end{defin}

For shortness, we will write \textbf{OGC} for ``orthogonal
geodesic chord'' and \textbf{WOGC} for ``weak orthogonal geodesic
chord''.

In the central result of this paper we will give a lower estimate
on the number of distinct orthogonal geodesic chords; we recall here
some results in this direction available in the literature.
In \cite{bos}, Bos proved that if $\partial\Omega$ is smooth,
$\overline\Omega$ convex and homeomorphic to the $m$-dimensional
disk, then there are at least $m$ distinct OGC's for $\overline\Omega$.
Such a result is a generalization
of a classical result by Ljusternik and Schnirelman (see \cite{LustSchn}),
where the same result was proven for convex subsets of $\R^m$
endowed with the Euclidean metric. Bos' result was used in \cite{gluckziller}
to prove a multiplicity result for brake orbits under a certain ``non-resonance condition''.
Counterexamples show that, if one drops the convexity assumption, the lower estimate for orthogonal
geodesic chords given in Bos' theorem does not hold.

Motivated by the study of a certain class of Hamiltonian systems (see Subsection \ref{sub:brakehom}),
in this paper we will study the case of sets with strongly concave boundary.
A natural conjecture is that, also in the concave case, one should have at least $m$ distinct orthogonal geodesic chords in an $m$-disk, but at this stage, this seems to be a quite hard result to prove.
Having this goal in mind, in this paper we give a positive answer to our conjecture in the special case when $m=2$.
Our central result is the following:
\begin{teo}\label{thm:main}
Let $\Omega$ be an open subset of $M$ with smooth boundary
$\partial\Omega$, such that $\overline\Omega$ is strongly concave
and homeomorphic to the $m$--dimensional disk.
Then, there are at least two geometrically distinct\footnote{%
By {\em geometrically distinct\/} curves we mean
curves having distinct images as subsets of $\overline\Omega$.}
orthogonal geodesic chords in $\overline\Omega$.
\end{teo}

A similar multiplicity result was proved in \cite{arma}, assuming that $\overline\Omega$ is
homeomorphic to the $m$--dimensional annulus.\smallskip

\subsection{Reduction to the case without WOGC}
Although the general class of weak orthogonal geodesic
chords are perfectly acceptable solutions of our initial geometrical
problem, our suggested construction of a variational setup works well
only in a situation where one can
exclude {\em a priori\/} the existence in $\overline\Omega$ of
orthogonal geodesic chords $\gamma:[a,b]\to\overline\Omega$ for
which there exists $s_0\in\left]a,b\right[$ such that
$\gamma(s_0)\in\partial\Omega$.

One does not lose generality in assuming
that there are no such WOGC's in $\overline\Omega$ by recalling the
following result from \cite{GGP1}:

\begin{prop}
\label{thm:noWOGC}
Let $\Omega\subset M$ be an open set
whose boundary $\partial\Omega$ is smooth and compact and with $\overline\Omega$
strongly concave.
Assume that there are only a finite number of  (crossing) orthogonal
geodesic chords in $\overline\Omega$. Then, there exists an
open subset $\Omega'\subset\Omega$ with the following properties:
\begin{enumerate}
\item\label{itm:nowogcs1} $\overline{\Omega'}$ is diffeomorphic to $\overline\Omega$
and it has smooth boundary;
\item\label{itm:nowogcs2} $\overline{\Omega'}$ is strongly concave;
\item\label{itm:nowogcs3} the number of (crossing) OGC's in $\overline{\Omega'}$  is less than
or equal to the number of (crossing) OGC's in $\overline\Omega$ ;
\item\label{itm:nowogcs4} there are no (crossing) WOGC's in $\overline{\Omega'}$.
\end{enumerate}
\end{prop}
\begin{proof}
See \cite[Proposition~2.6]{GGP1}
\end{proof}

\begin{rem}\label{thm:remmain}
In view of the result of Proposition~\ref{thm:noWOGC}, it suffices to prove Theorem~\ref{thm:main} under the further
assumption that:
\begin{equation}\label{eq:ipotesi}
\text{ there are no WOGC's in }\overline\Omega.
\end{equation}
For this reason, we will henceforth assume \eqref{eq:ipotesi}.
\end{rem}

\subsection{On the curve shortening method in concave manifolds}
Multiplicity of OGC's in the case of compact manifolds having convex
boundary is typically proven by applying a curve-shortening argument.
From an abstract viewpoint, the curve-shortening process can be seen
as the construction of a flow in the space of paths, along whose trajectories
the length or energy functional is decreasing.

In this paper we will follow the same procedure, with the difference
that both the space of paths and the shortening flow have
to be defined appropriately.

Shortening a curve having image in a closed convex subset $\overline\Omega$
of a Riemannian manifold produces another curve in $\overline\Omega$; in
this sense, we think of the shortening flow as being ``inward pushing''
in the convex case.
As opposite to the convex
case, the shortening flow in the concave case will be ``outwards pushing'', and
this fact requires the one should consider only those portions of a curve
that remain inside $\overline\Omega$ when it is stretched outwards.
This type of analysis has been carried out in \cite{esistenza}, and we shall employ here many of the
results proved in \cite{esistenza}.

The concavity condition plays a central role in the variational setup of our
construction. ``Variational criticality'' relatively to the energy
functional will be defined in terms of ``outwards pushing'' infinitesimal
deformations of the path space (see Definition~\ref{thm:defvariatcrit}).
The class of variationally critical portions contains properly the set of portions
consisting of crossing OGC's; such curves will
be defined as ``geometrically critical'' paths (see Definition~\ref{thm:defgeomcrit}).
In order to construct the shortening flow, an accurate analysis of all
possible variationally critical paths is required (Section~\ref{sub:description}),
and the concavity condition will guarantee that such paths are
\emph{well behaved} (see Lemma~\ref{thm:leminsez4.4}, Proposition~\ref{thm:regcritpt}
and Proposition~\ref{thm:irregcritpt}).

Once that a reasonable classification of variationally critical points
is obtained, the shortening flow is constructed
by techniques which are typical of  pseudo-gradient vector field approach.
The crucial property of the shortening procedure is that its flow lines move away from critical
portions which are not OGC's, in the same way that the integral line of a pseudo-gradient vector field
move away from points that are not critical.
A  technical description of the abstract
\emph{minimax} framework that we will use is given in Subsection~\ref{sub:discussion}.
\end{subsection}

\begin{subsection}{Brake and Homoclinic Orbits of Hamiltonian Systems}
\label{sub:brakehom}
The result of Theorem~\ref{thm:main} can be applied to prove a multiplicity
result for brake orbits and homoclinic orbits, as follows.

Let $p=(p_i)$, $q=(q^i)$ be coordinates on $\R^{2m}$, and
let us consider a {\em natural\/} Hamiltonian function
$H\in C^2\big(\R^{2m},\R\big)$, i.e., a function of the form
\begin{equation}\label{eq:hamfun}
H(p,q)=\frac 12 \sum_{i,j=1}^m a^{ij}(q)p_ip_j+V(q),
\end{equation}
where $V\in C^2\big(\R^{m},\R\big)$ and $A(q)=\big(a^{ij}(q)\big)$ is
a positive definite quadratic form on $\R^m$:
\[\sum_{i,j=1}^m a^{ij}(q)p_ip_j\ge\nu(q)\vert q\vert^2\]
for some continuous function $\nu:\R^m\to\R^+$ and for all $(p,q)\in\R^{2m}$.

The corresponding Hamiltonian system is:
\begin{equation}\label{eq:HS}
\left\{
\begin{aligned}
&\dot p=-\frac{\partial H}{\partial q}\\
&\dot q=\frac{\partial H}{\partial p},
\end{aligned}
\right.
\end{equation}
where the dot denotes differentiation with respect to time.

For all $q\in\R^m$, denote by $\mathcal L(q):\R^m\to\R^m$ the linear
isomorphism whose matrix with respect to the canonical basis is
$\big(a_{ij}(q)\big)$, which is the inverse of $\big(a^{ij}(q)\big)$;
it is easily seen that, if $(p,q)$ is
a solution of class $C^1$ of \eqref{eq:HS}, then
$q$ is actually  a map of class $C^2$
and
\begin{equation}\label{eq:pintermsofq}
p=\mathcal L(q)\dot q.
\end{equation} With a
slight abuse of language, we will say that a
$C^2$-curve $q:I\to
\R^m$ ($I$ interval in $\R$) is a solution of \eqref{eq:HS} if $(p,q)$ is a solution
of \eqref{eq:HS} where $p$ is given by \eqref{eq:pintermsofq}.
Since the system \eqref{eq:HS} is autonomous, i.e., time independent,
then the function $H$ is constant along each solution, and it represents
the total energy of the solution of the dynamical system. There
exists a large amount of literature concerning the study of periodic
solutions of autonomous Hamiltonian systems having energy $H$ prescribed
(see for instance \cite{LiuLong,LZ,Long,rab} and the references therein).

\subsection{The Seifert conjecture in dimension $2$}
We will be concerned with
a special kind of periodic solutions of \eqref{eq:HS}, called {\em brake orbits}.
A brake orbit for the system
\eqref{eq:HS} is a non-constant periodic solution $\R\ni t\mapsto\big(p(t),q(t)\big)
\in\R^{2m}$
of class $C^2$ with the property that $p(0)=p(T)=0$ for
some $T>0$. Since $H$ is even in the variable $p$, a brake orbit
$(p,q)$ is $2T$-periodic, with $p$ odd and $q$ even about $t=0$ and about $t=T$.
Clearly, if $E$ is the energy of a brake orbit $(p,q)$, then
$V\big(q(0)\big)=V\big(q(T)\big)=E$.

The link between solutions of brake orbits and orthogonal geodesic chords
is obtained in \cite[Theorem~5.9]{GGP1}.
Using this theorem and Theorem~\ref{thm:main}, we get immediately the following:
\begin{teo}\label{thm:1.7}
Let $H\in C^2\big(\R^{2m},\R\big)$ be a natural Hamiltonian function
as in \eqref{eq:hamfun}, $E\in\R$ and \[\Omega_E=V^{-1}\big(\left]-\infty,E\right[\big).\]
Assume that $\mathrm dV(x)\ne0$ for all $x\in\partial\Omega_E$ and that
$\overline{\Omega}_E$ is homeomorphic to a $m$-disk.
Then, the Hamiltonian system \eqref{eq:HS} has at least two geometrically
distinct brake orbits  having energy $E$.
\end{teo}

Multiplicity results for brake orbits in even, convex case are obtained e.g. in \cite{LZ,LZZ,Z1,Z2,Z3}.
\medskip

In \cite{seifert}, it was conjectured by Seifert the existence of at least $m$ brake orbits and it is well known that such lower estimate for the number of brake orbits
cannot be improved.
Indeed, consider the Hamiltonian:
\[
H(q,p)=\tfrac12|p|^2+\sum_{i=1}^m\lambda_i^2 q_i^2,\qquad (q,p)\in\R^{2m},
\]
where $\lambda_i\not=0$ for all $i$. If $E>0$ and the squared ratios
$\left({\lambda_i}/{\lambda_j}\right)^2$ are irrational for all $i\ne j$,
then the only periodic solutions of \eqref{eq:HS} with energy
$E$ are the $m$ brake orbits moving along the axes of the ellipsoid with equation
\[
\sum_{i=1}^m\lambda_i^2q_i^2=E.
\]
The result in \cite{LZ} is a proof of the Seifert conjecture (in any dimension) under the assumption that the potential is convex and even. Theorem~\ref{thm:1.7} gives a proof of the Seifert conjecture in dimension $m=2$, without any assumption on the potential.
\end{subsection}

\section{Main ideas of the proof}\label{main}
In this section we will give an outline of the paper, describing the functional framework and the main ideas of the proofs.

\subsection{Presentation of the proof of Theorems~\ref{thm:main} and \ref{thm:1.7}}\label{sec:outline}
The proof of our multiplicity result will be carried out in the
following way. Set $W = \{x \in \R^2: V(x) < E\}$.

\begin{itemize}
\item
Using the well known Maupertuis--Jacobi variational principle,
see e.g. \cite[Proposition 4.1]{GGP1}, brake orbits for the
given Hamiltonian system are characterized, up to a
reparameterization, as geodesics with endpoints on
$\partial W$ relatively to a certain Riemannian metric, the
so--called Jacobi metric on $W$, singular on $\partial W$ given by
$g_*(v,v)=\big(E-V(x)\big)g_0(v,v)$,
where $g_0(v,v) = \frac 12\sum_{i,j=1}^m
a_{ij}(x)\,v_i v_j$;\smallskip

\item by means of the Jacobi metric and the
induced "distance from the boundary" function, one gets rid of
the metric singularity on the boundary, and the problem is
reduced to the search of geometrically distinct geodesics,
orthogonal to the boundary of a Riemannian manifold which is
homeomorphic to the $m$--dimensional, whose boundary satisfies a strong
\emph{concavity} condition, cf \cite{GGP1};\smallskip

\item a minimax argument will be applied to a suitable class of homotopies and to a particular nonsmooth
functional (for the classical minimax theory cf e.g. \cite{MW, struwe}).
\end{itemize}

\subsection{Abstract Ljusternik--Schnirelman theory}\label{sub:discussion}

For the minimax theory  we shall use the following topological invariant. Consider a topological space $X$ and $Y \subset X$.
We shall use a suitable version of the
relative category in $\Xcal\text{\ mod\ }\Ycal$ (see \cite{FH,FW})
as topological invariant, which is defined as follows.

Let $\Dcal\subset\Xcal$ be a
closed subset, and assume that there exists $k>0$ and  $A_0,A_1,\ldots, A_k$ open subsets of $\Xcal$ such that:
\begin{itemize}
\item[(a)] $\Dcal\subset\bigcup_{i=0}^k A_i$;
\item[(b)] for any $i=1,\ldots,k$ there
exists a homotopy $h_i$ sending $A_i$ to a single point moving
in $\Xcal$, while the homotopy $h_0$
sends $A_0$ inside $\Ycal$ moving $A_0 \cap\Ycal$ in $\Ycal$.
 \end{itemize}
The minimal integer $k$ with the above
properties is the relative category of $\Dcal$ in  $\Xcal\text{\ mod\ }\Ycal$
and it will be denoted by $\cat_{\Xcal,\Ycal}(\Dcal)$. We shall use it with 
$\Xcal = \mathbb{S}^{m-1}\times \mathbb{S}^{m-1}$ and $\Ycal=\{(A,B)\in \mathbb{S}^{m-1} \,:\,A=B\}\equiv \Delta^{m-1}$, where
$\mathbb{S}^{m-1}$ is the $(m-1)$--dimensional unit sphere.

In \cite{esistenza} a different relative category is considered. There, the maps  $h_i$ were assumed to send the $A_i$'s to a single point moving outside the set $\Delta^{m-1}$; moreover, it was used a quotient of the product $\mathbb S^{m-1}\times\mathbb S^{m-1}$ obtained by identifying the pairs $(A,B)$ and $(B,A)$. Its numerical value is $m$, but unfortunately this notion of relative category is not compatible with the definition of the functional $\mathcal F$ used in the minimax argument, and no multiplicity result can be obtained. In order to have a relative category which fits with the properties of the functional $\mathcal F$, one must relax the assumptions on the maps $h_i$, and require that they take values in all the space $\mathbb{S}^{m-1}\times \mathbb{S}^{m-1}$. This gives a lower numerical value for such new notion of relative category, which is less than or equal to $2$, as we can see by the same topological arguments used in \cite{G}. This suggests that it is more convenient to use a  relative category without the symmetry given by the identification of the pairs $(A,B)$ and $(B,A)$. With this definition, we have:

\begin{lem}\label{thm:estimatecat}
For any $m \geq2$, $\cat_{\Xcal,\Ycal}(\Xcal)\geq 2$.
\end{lem}
The proof of Lemma~\ref{thm:estimatecat} uses the notion of cuplength in cohomology, and it will be given in Appendix~\ref{sec:appendixrelLScat}.
Note that, in fact, the Lemma~\ref{thm:estimatecat} implies the equality $\cat_{\Xcal,\Ycal}(\Xcal)=2$, as it easy to show that, in any dimension, $\cat_{\Xcal,\Ycal}(\Xcal) \leq 2$.

\medskip

The problem of  finding orthogonal geodesic chords in a domain
$\overline\Omega$ of a Riemannian manifold $M$ with non-convex
boundary $\partial\Omega$ cannot be cast in a standard smooth
variational context, due mainly to the fact that the classical
shortening flow on the set of curves in $\overline\Omega$ with
endpoints on the boundary
produces stationary curves that are not "classical geodesics". In
order to overcome this problem, our strategy will be to reproduce
the ``ingredients'' of the classical smooth theory in a suitable non-smooth context. More precisely, we will define the following
objects:
\begin{itemize}
\item a metric space $\Mcal$, that consists of curves of class $H^{1}$
having image in an open neighborhood of $\overline\Omega$ in
$\Mcal$, and whose endpoints remain outside $\Omega$;
\item a compact  subset $\mathfrak C$ of $\Mcal$ which is
homeomorphic to the set of chords in the unit disk $\mathbb D^m$ with both
endpoints in $\mathbb{S}^{m-1}$ (and therefore homeomorphic to
$\mathcal X = \mathbb{S}^{m-1}\times\ \mathbb{S}^{m-1}$));
\item the class of the closed $\Rcal$--invariant subsets $\mathcal
D$ of $\mathfrak C$;
\item a family ${\mathcal H}$ consisting of pairs $(\mathcal D,h)$, where
$\mathcal D$ is a closed subset of $\mathfrak C$ and
$h:[0,1]\times\mathcal D\to\Mcal$ is a homotopy whose properties will be described in section
\ref{sec:homotopies};
\item a functional $\mathcal F :{\mathcal H}\rightarrow \R^{+}$, constructed starting from the classical
energy functional used for the geodesic problem.
\end{itemize}

We will define suitable notions of critical values for the
functional $\mathcal F$, in such a way that distinct critical values
determine geometrically distinct orthogonal geodesic chords in $\overline\Omega$.

Denote by $\star$ the operation of \emph{concatenation
of homotopies}, see \eqref{eq:defconcatenationhomotop}.
We shall say that a real number $c$ is a \emph{topological regular value} of $\mathcal F$ if there exists
$\bar\varepsilon>0$ such that for all $(\mathcal D,h)\in\mathcal H$
satisfying $\mathcal F(\mathcal D,h)) \leq c + \bar\varepsilon$ there exists
a homotopy $\eta$ such that $(\mathcal D,\eta \star h)\in\mathcal H$,
satisfying
\[
\mathcal F(\mathcal D,\eta \star h)\le c-\bar\varepsilon.
\]

A \emph{topological critical value} of $\mathcal F$ is a real number which is not a regular value.
\smallskip

Once this set up has been established, the proof of multiplicity
of critical points of $\mathcal F$ is carried out along the lines
of the standard relative Ljusternik--Schnirelman theory, as follows. Denote by $\mathfrak C_0$ the set of constant curves in
$\mathfrak C$ (which is homeomorphic to $Y$).
For $i=1,2$, set:
\begin{equation}\label{eq:defGammai}
\Gamma_i=\big\{\mathcal D\in\mathfrak C:\; \mathcal D
\,\text{is closed\ },\cat_{{\mathfrak C},{\mathfrak C_0}}
({\mathcal D})\ge i\big\},
\end{equation}
and define
\begin{equation}\label{eq:defci}
c_i=\inf_{\substack{\mathcal D\in\Gamma_i\\  (\mathcal D,h)\in\mathcal H}}
\mathcal F(\mathcal D,h).
\end{equation}
As observed in Remark \ref{rem:6.6}, $\mathcal H$ is not empty since $(\mathfrak C,\mathrm{I}_{\mathfrak C})$
belongs to the class $\mathcal H$,
where we denote by $\mathrm{I}_{\mathfrak C}:[0,1]\times\mathfrak C\to\mathfrak C$  the map $\mathrm{I}_{\mathfrak C}(\tau,x)=x$
for all $\tau$ and all $x$.
Moreover, by Lemma \ref{thm:estimatecat}, $\mathfrak C \in \Gamma_i$ for any $i=1,2$, and from this  we deduce that any $c_i$ is a finite real number.

By the very definition, one sees immediately that each $c_i$ is a topological critical value of
$\mathcal F$; moreover, since $\Gamma_{1}\subset\Gamma_2$, we have $c_1 \leq c_2$.

The crucial point of the construction is the proof of some ``deformation lemmas''
for the sublevels of $\mathcal F$ using the homotopies in $\mathcal H_{1}$,
in order to obtain that the $c_i$'s are energy values of geometrically distinct orthogonal geodesic chords parameterized in $[0,1]$.

The first deformation lemma
tells us that the topological critical values of $\mathcal F$ correspond to orthogonal geodesic chords,
in the sense that if $c$ is a topological critical value for $\mathcal F$
then it is a \emph{geometrical critical value} (cf. Definition~\ref{thm:defgeomcrit}): there exists an orthogonal geodesic chord
$\gamma$ (parameterized in the interval $[0,1]$)
such that $\frac12\int_0^1 g(\dot \gamma,\dot \gamma)ds = c$.
Indeed if $c > 0$ is not a geometrical critical value, there exists $\epsilon > 0$
such that for any $(\mathcal D,h) \in \mathcal H$ satisfying
$\mathcal F(\mathcal D,h) \leq c+\epsilon$, there exists a homotopy $\eta$
such that $(\mathcal D, \eta \star h) \in \mathcal H$ and $\mathcal F(\mathcal D,\eta\star h) \leq c-\epsilon$ (cf. \ref{thm:firstdeflemma}).

The second deformation lemma (cf. \ref{thm:prop9.3})
says that a similar deformation
exists also for geometrical critical values, provided
that a suitable contractible neighborhood is removed.

More precisely, in our case, given a geometrical critical value $c > 0$, assuming there is only a finite number of
orthogonal geodesic chord in $\overline\Omega$ having energy $c$, we will prove
the existence of $\bar\varepsilon>0$ such that, for all $(\mathcal D,h)\in\mathcal H$ with
$\mathcal F(\mathcal D,h)\le c+\bar\varepsilon$ there exists an open subset $\mathcal A\subset\mathfrak C$
and a homotopy $\eta$ such that:
\begin{itemize}
\item[(i)] $(\mathcal D\setminus\mathcal A,\eta \star h)\in\mathcal H$;
\item[(ii)] $\mathcal F\big(\mathcal D\setminus\mathcal A,\eta \star h\big)\le c-\bar\varepsilon$;
\item[(iii)] $A$ is contractible in ${\mathfrak C}$
(hence, $\cat_{\Xcal,\Ycal}(\Dcal\setminus{\mathcal A})\ge
\cat_{\Xcal,\Ycal}(\Dcal)-1$).
\end{itemize}
 
Moreover, we also see that low sublevels of the functional $\mathcal F$ consist of curves that can be deformed on $\partial \Omega$, obtaining that $c_i>0$ for any $i=1,2$,
while by the two Fundamental Deformations Lemmas above we have:
\begin{itemize}
\item[(a)] $c_i$ is a geometrical critical value;
\item[(b)] $c_1<c_2$, assuming the existence of only a finite number of orthogonal geodesic
chords in $\overline\Omega$.
\end{itemize}
Note that if  $c=c_1=c_2$ we can get a contradiction in the following way. Choose $\bar\varepsilon>0$ as
in the second deformation Lemma, and take $(\mathcal D,h)\in\mathcal H$ such that $\mathcal D\in\Gamma_2$
and $\mathcal F(\mathcal D,h)\le c_2+\bar\varepsilon$. Let $\mathcal A\subset\mathfrak C$ and
$\eta$ be as above.
Then, $\mathcal F(\mathcal D\setminus\mathcal A,\eta \star h)\le c_1-\bar\varepsilon$, which is absurd, because
$\mathcal D\setminus\mathcal A\in\Gamma_1$ and $(D\setminus\mathcal A,\eta \star h) \in \mathcal H$.

The argument proves the existence of at least 2 distinct geometrical
critical values; the crucial point is that distinct geometrical
critical values produce geometrically distinct orthogonal geodesic
chords (cf. Proposition \ref{thm:distinct}). Then, using the results in \cite{GGP1}, we obtain the existence of at least two geometrically distinct brake orbits.

\section{The functional framework}
\label{sub:varframe}

Throughout the paper, $(M,g)$ will denote a Riemannian manifold
of class $C^2$ having dimension $m$;
all our constructions will be made in suitable (relatively) compact
subsets of $M$, and for this reason it will not be restrictive
to assume, as we will, that $(M,g)$ is complete.
Furthermore, we will work mainly in open subsets $\Omega$ of $M$ whose closure
is  homeomorphic to a $m$--dimensional disk, and
in order to simplify the exposition we will assume that, indeed,
$\overline\Omega$ is embedded topologically in $\R^m$, which will
allow to use an auxiliary linear structure in a  neighborhood
of $\overline\Omega$.
We will also assume that $\overline\Omega$ is strongly concave in $M$.

The symbol $H^1\big([a,b],\R^m\big)$ will denote the Sobolev
space of all absolutely continuous curves in $\R^m$
whose weak derivative is square integrable. Similarly,
$H^1\big([a,b],\R^m\big)$ will denote the infinite dimensional
Hilbert manifold consisting of all absolutely continuous
curves $x:[a,b]\to M$ such that $\varphi\circ x\vert_{[c,d]}\in H^1\big([c,d],\R^m)$
for all chart $\varphi:U\subset M\to \R^m$ of $M$ such that $x\big([c,d]\big)
\subset U$. By $H^1_0\big(\left]a,b\right[,\R^m\big)$
we will denote the subset of $H^1\big([a,b],\R^m\big)$ with $x(a)=x(b)=0$.
For $A \subset \R^m$ and $a < b$ we set
\begin{equation}\label{eq:defH1abA}
H^1\big([a,b],A\big)=\big\{x\in H^1\big([a,b],\R^m\big):x(s)\in A\ \text{for
all $s\in[a,b]$}\big\}.
\end{equation}
The Hilbert space norm $\Vert\cdot\Vert_{a,b}$ of $H^1\big([a,b],\R^m\big)$
(equivalent to the usual one)
will be defined by:
\begin{equation}\label{eq:norma-ab}
\Vert x\Vert_{a,b}=\left(\frac{\Vert x(a)\Vert_E^2+\int_a^b\Vert\dot x(s)\Vert^2_E\,\text ds}2\right)^{1/2},
\end{equation}
where $\Vert\cdot\Vert_E$ is the Euclidean norm in $\R^m$. Note that by \eqref{eq:norma-ab}
\begin{equation}\label{eq:stima1.8}
\Vert x\Vert_{L^\infty([a,b],\R^m)}\le\Vert x\Vert_{a,b},
\end{equation}
and this simplifies some estimates in the proofs of the deformation Lemmas (cf. \cite{esistenza}).
We shall use also the space $H^{2,\infty}$
which consists of differentiable curves with absolutely continuous derivative
and having bounded weak second derivative.

\begin{rem}\label{rem:rep}
In the development of our results, we will consider
curves $x$ with variable domain $[a,b]\subset[0,1]$.
In this situation, by $H^1$-convergence of a sequence $x_n:[a_n,b_n]\to M$
to a curve $x:[a,b]\to M$ we will mean that $a_n$ tends to $a$, $b_n$ tends
to $b$ and  $\widehat x_n:[a,b]\to M$
is $H^1$-convergent to $x$ in $H^1\big([a,b],M\big)$ as $n\to\infty$, where
$\widehat x_n$ is the unique affine reparameterization of $x$ on the interval
$[a,b]$. One defines similarly the notion of $H^1$--weak convergence and of uniform convergence
for sequences of curves with variable domain.
\end{rem}

It will be useful also to consider
the  flows $\eta^{+}(\tau,x)$ and $\eta^{-}(\tau,x)$ on the Riemannian manifold $M$ defined by
\begin{equation}\label{eq:flusso+}
\left\{\begin{array}{l}\dfrac{\mathrm d{\eta^+}}{\mathrm d\tau}(\tau)=
\dfrac{\nabla \phi(\eta^+)}{\Vert \nabla \phi(\eta^+) \Vert^{2}} \\[.5cm]
{\eta}^{+}(0)=x \in\big \{y \in M: -\delta_0 \leq \phi(y) \leq \delta_0\big\}, \end{array} \right.
\end{equation}

and

\begin{equation}\label{eq:flusso-}
\left\{\begin{array}{l}\dfrac{\mathrm d{\eta}^{-}}{\mathrm d\tau}(\tau)=
\dfrac{-\nabla \phi(\eta^-)}{\Vert \nabla \phi(\eta^-) \Vert^{2}} \\[.5cm]
{\eta}^{-}(0)=x \in\big \{y \in M: -\delta_0 \leq \phi(y) \leq \delta_0\big\}, \end{array} \right.
\end{equation}
where $\Vert \cdot \Vert$ is the norm induced by $g$.

\begin{rem}\label{rem:flusso-well-defined}
Note that $\eta^{+}(\tau,x)$ and $\eta^{-}(\tau,x)$ are well defined, because $\nabla \phi \not=0$ on the strip $\phi^{-1}\big([-\delta_0,\delta_0]\big)$. Moreover, using $\eta^+$ and $\eta^-$ we can show that the exists a homeomorphism between $\phi^{-1}\big([-\delta_0,\delta_0]\big)$ and $\big\{y \in \mathbb{R}^m: 1-\delta_0 \leq \Vert y \Vert_E \leq 1+\delta_0\big\}$. Therefore it must be $\delta_0 < 1$, since $\overline\Omega$ is homeomorphic to the unit $m$--dimensional disk.
\end{rem}
\bigskip

Now, fix a convex $C^2$--real map $\chi$ defined in $[0,1+\delta_0,1]$ such that
$\chi(s)=s-1$ for any $s \in [1-\delta_0,1+\delta_0]$, $\chi'(s) > 0$ for any $s \in [0,1+\delta_0]$, $\chi''(0)=0$ and consider the map:
\begin{equation}\label{eq:phisuldisco}
\phi_{\mathbb{D}^m}(z) = \chi(\Vert z \Vert_E),
\end{equation}
where ${\mathbb{D}^m}$ denotes the m--dimensional disk.
Note that $\phi_{\mathbb{D}^m}$ satisfies the properties of the map $\phi$ described in Remark \ref{thm:remphisecond} for the set $\overline \Omega = {\mathbb{D}^m}$ and the Riemann structure given by the Euclidean metric.\smallskip

We have the following
\begin{lem}\label{thm:riconduco-sfera}
There exists a homeomorphism $\Psi: \phi^{-1}(]-\infty,\delta_0]) \rightarrow \phi_{\mathbb{D}^m}^{-1}(]-\infty,\delta_0])$, which is
of class $C^1$ on $\phi^{-1}([-\delta_0,\delta_0])$,
such that
\begin{equation}\label{eq:corrispondenza}
-\phi(y)= 1-\Vert \Psi(y) \Vert_E \quad\forall\;y \in \phi^{-1}([-\delta_0,\delta_0]).
\end{equation}
\end{lem}

\begin{proof}
Consider any  homeomorphism $\psi : \overline\Omega \rightarrow \mathbb{D}^m$  and the flow $\eta^-$ given in \eqref{eq:flusso-}. Note that for all $y \in \phi^{-1}([-\delta_0,0])$ there exists a unique $y_0 \in \partial \Omega$ and $\tau \in [-\delta_0,0]$ such that
\[
y = \eta^{-}(\tau,y_0).
\]
For any $y \in \phi^{-1}([-\delta_0,0])$ we set
\[
\psi_0^-(y)=(1+\tau)\psi(y_0),
\]
obtaining a diffeomorphism between
$\phi^{-1}([-\delta_0,0])$ and $\phi_{\mathbb{D}^m}^{-1}([-\delta_0,0])$.
Similarly, using the flow  $\eta^+$ starting from $\partial\Omega$, we can define $\psi_0^+$ on $\phi^{-1}([0,\delta_0])$ such that $\psi_0^+(y)=\psi_0^-(y)$ on $\partial \Omega$,
obtaining a diffeomorphism $\psi_0$ between $\phi^{-1}([-\delta_0,\delta_0])$ and $\phi_{\mathbb{D}^m}^{-1}([-\delta_0,\delta_0])$. Now, we just have to extend
$\psi_0$ as homeomorphism to all $\phi^{-1}(]-\infty,\delta_0])$.

Towards this goal, set
\[
P_0 = \psi_0^{-1},
\]
which is well defined on $\phi_{\mathbb{D}^m}^{-1}([-\delta_0,\delta_0])$,
\[
\widehat{P_0} = P_{0}\big|_{\phi_{\mathbb{D}^m}^{-1}(-\delta_0)},
\]
and
\[
Q = \psi\vert_{\phi^{-1}(-\delta_0)} \circ \widehat{P_0}
\]
which is an homeomorphism on $\phi_{\mathbb{D}^m}^{-1}(-\delta_0)$.

Now extend $Q$ to all $\big\{z \in \mathbb{R}^m: \Vert z \Vert_E \leq 1-\delta_0\big\}$ by setting:
\[
\tilde Q(z) =
\begin{cases}
\dfrac{\Vert z \Vert_E}{1-\delta_0}Q\big(\frac{1-\delta_0}{\Vert z \Vert_E}z\big)& \text{ if } z \neq 0\\
0& \text{ if } z = 0.
\end{cases}
\]
Finally, the desired homeomorphism $\Psi$ is obtained by setting:
\[
\Psi^{-1}(z)=
\begin{cases}
\psi_0^{-1}(z) &\text{ if } z \in \phi_{\mathbb{D}^m}^{-1}([-\delta_0,\delta_0])\\
\psi^{-1}(\tilde Q(z)) &\text{ if }z \in \phi_{\mathbb{D}^m}^{-1}(]-\infty,-\delta_0]).
\end{cases}
\]
\qedhere
\end{proof}

Throughout the paper we shall use also the following constant:
\begin{equation}\label{eq:1.9fabio}
K_0=\max_{x\in\phi^{-1}(]-\infty,\delta_0]}\Vert\nabla\phi(x)\Vert.
\end{equation}

\subsection{Path space and maximal intervals}
\label{sec:pathspace}
In this subsection we will describe the set of curves $\mathfrak M$, which will be the ambient space of our minimax framework,
and the set $\mathfrak C \subset \mathfrak M$ homeomorphic
to  $\mathbb{S}^{1} \times \mathbb{S}^{1}$, that encodes all the topological information about $\mathfrak M$.
\smallskip

Let $\delta_0>0$ be as in Remark \ref{thm:remopencondition}.
Consider first the following set of paths
\begin{equation}\label{eq:defM}
\mathfrak M_0=\Big\{x\in
H^1\big([0,1],\phi^{-1}(\left]-\infty,\delta_0\right[)\big):
\phi(x(0))\ge 0,\,\phi(x(1))\ge 0\Big\},
\end{equation}
see Figure \ref{fig:3bis}.
\begin{figure}
\begin{center}
\psfull \epsfig{file=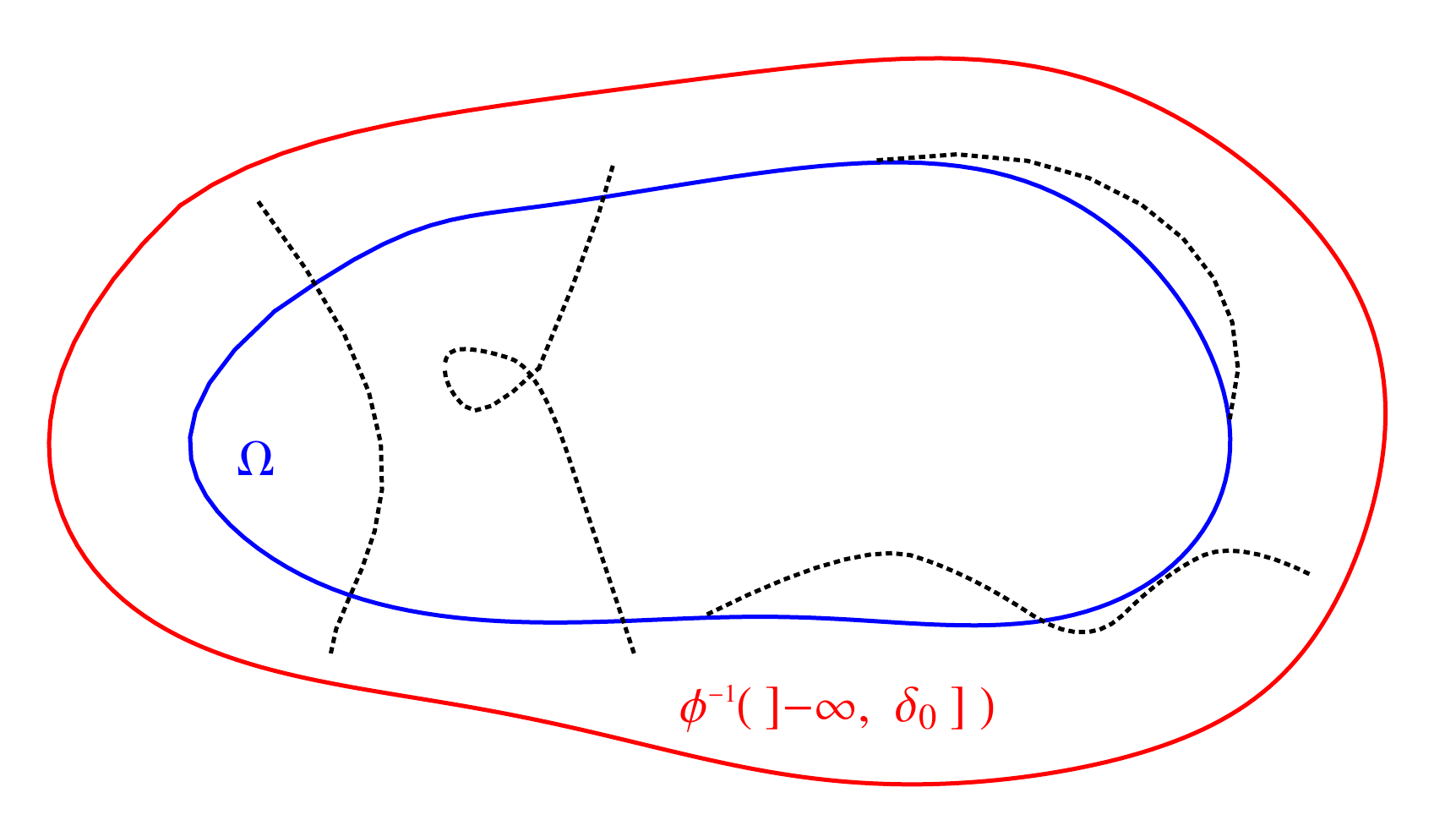, height=5cm} \caption{
Curves (the dotted lines) representing typical elements of the path space $\mathfrak
M_0$.}\label{fig:3bis}
\end{center}
\end{figure}
 
This is a subset of the Hilbert space $H^1\big([0,1],\R^m\big)$, and
it will be topologized with the induced metric.

The following result will be used systematically throughout the
paper:
\begin{lem}\label{thm:lemmacazzatina}
If $x\in\mathfrak M_0$ and $[a,b]\subset[0,1]$ is such that
$x(a)\in\partial\Omega$ and there exists $\bar s\in[a,b]$ such that
$\phi(x(\bar s))\le-\delta<0$, then:
\begin{equation}\label{eq:2.9fabio}
b-a\ge\frac{\delta^2}{K_0^2}\left(\int_a^bg(\dot x,\dot x)\,\mathrm
d\sigma\right)^{-1} ,
\end{equation}
and
\begin{equation}\label{eq:aggiuntalemma2.1}
\sup\{|\phi(x(s)|:s \in [a,b]\} \leq \sqrt2K_0\left(\frac{b-a}2 \int_a^bg(\dot x,\dot x) \,
\mathrm d\sigma\right)^{\frac12}
\end{equation}
where $K_0$ is defined in \eqref{eq:1.9fabio}.
\end{lem}
\begin{proof}
Since $\phi(x(a))=0$ we have, for any $s\in [a,b]$:
\begin{multline*}
\vert \phi(x(s))\vert = \vert\phi(x(s))-\phi(x(a))\vert\le\int_a^{s}\vert g\big(\nabla\phi(x(\sigma)),\dot x(\sigma)\big)\vert\,
\mathrm d\sigma \le \\ \int_a^b \vert
g\big(\nabla\phi(x(\sigma)),\dot x(\sigma)\big)\vert\, \mathrm d\sigma
 \le K_0\int_a^bg(\dot x,\dot x)^{\frac12}  \mathrm
d\sigma\\
\le K_0\sqrt{b-a} \left(  \int_a^bg(\dot x,\dot x) \,\mathrm
d\sigma\right)^{\frac12},
\end{multline*}
from which \eqref{eq:aggiuntalemma2.1} follows. Moreover, the same estimate shows that, if there exists $\bar s\in[a,b]$ such that
$\phi(x(\bar s))\le-\delta<0$, then \eqref{eq:2.9fabio} holds.
\end{proof}

For all $x\in\mathfrak M_0$, let $\mathcal I_{x}^{0}$ and  $\mathcal I_x$
denote the following collections of closed subintervals of $[0,1]$:
\[
\mathcal I_{x}^{0} = \big\{[a,b]\subset [0,1]: x([a,b]) \in
\overline\Omega, x(a),x(b) \in \partial\Omega \big\},
\]
\[
\mathcal I_x=\big\{[a,b]\in \mathcal I_{x}^{0}
\text{ and $[a,b]$ is maximal with respect to this
property}\big\}.
\]

\bigskip

\begin{rem}\label{simple-sc}
It is immediate to verify the following semicontinuity property.
Suppose $x_n \rightarrow x$ in $\mathfrak M$, $[a,b] \in {\mathcal I}_x$ and $[a_n,b_n] \in {\mathcal I}_{x_n}$ with $[a_n,b_n] \cap [a,b] \neq \emptyset$
for all $n$. Then
\[a \leq \liminf_{n \rightarrow \infty}a_n \leq \limsup_{n \rightarrow \infty}b_n \leq b.\]
\end{rem}

\begin{rem}\label{rem:noROGC}
Note that if $\gamma: [0,1]\rightarrow \overline\Omega$ is an OGC, then $\gamma\not\equiv \gamma$. Indeed if by contradiction $\gamma(1-t)=\gamma(t)$ for any $t$, from which we deduce $\dot \gamma(\frac12)=0$ and by the conservation law of the energy we should have that $\gamma$ is constant.
\end{rem}

The following Lemma allows to describe the subset $\mathfrak C$ of ${\mathfrak M}_{0}$
which carries on  all the topological properties of ${\mathfrak M}_{0}$.

\begin{lem}\label{thm:corde} There exists there a continuous map $G:\partial\Omega\times\partial\Omega\to H^1([0,1],\overline\Omega)$ such that
\begin{enumerate}
\item\label{corde1} $G(A,B)(0)=A,\,\,G(A,B)(1)=B$.
\item\label{corde2} $A\not=B\,\Rightarrow\, G(A,B)(s)\in\Omega\,\forall s\in ]0,1[$.
\item\label{corde3} $G(A,A)(s)=A\,\forall s\in [0,1]$.
\item\label{corde5} Suppose that there exists $s_0 \in [0,1] : \phi(G(A,B)(s_0)) > -\delta_0$. Then the set $\{s \in [0,1]: \phi(G((A,B)(s))) \in [-\delta_0,0]\}$ consists of two intervals where \\ $\phi(G(A,B)(\cdot))$ is strictly monotone.
\end{enumerate}
\end{lem}

\begin{proof}
Let $\Psi: \phi^{-1}([-\infty,\delta_0]) \rightarrow \phi_{\mathbb{D}^m}^{-1}([-\infty,\delta_0])$ be the homeomorphism of Lemma \ref{thm:riconduco-sfera}.
Define
\[
\hat G(A,B)(s) = \Psi^{-1}\big((1-s)\Psi(A)+s\Psi(B)\big),\,A,B\in \overline\Omega.
\]

In general, if $\overline\Omega$ is only homeomorphic to the
disk  $\mathbb D^m$, the above definition produces curves
that in principle are only continuous. In order to produce curves
with an $H^1$-regularity, we use a broken geodesic approximation
argument.
Towards this goal note that if the curve
\[
(1-s)\Psi(A) + s\Psi(B)
\]
intersects $\phi_{\mathbb{D}^m}^{-1}(-\delta_0)$
this happen at the instants
\[
0<s_A \leq s_B < 1,
\]
with $s_A, s_B$ depending continuously by $A,B$ respectively.

Denote by $\varrho(\overline\Omega,g)$ the
infimum of the injectivity radii of all points of
 $\overline\Omega$ relatively
to the metric $g$ (cf. \cite{docarmo}). By compactness, there exists
$N_0\in\N$ with the property that
$\dist\big(\hat G(A,B)(a),\hat G(A,B)(b)\big)\le\varrho(\overline\Omega,g)$ whenever
 $\vert a-b\vert\le\frac1{N_0}$ (where $\dist$ denotes the distance induced by $g$).

Finally, for all $\hat G(A,B)$, denote by
$\gamma_{A,B}$ the broken geodesic obtained as concatenation of the
curves $\gamma_k:[s_A+\frac{k-1}{N_0}(s_B-s_A),s_A+\frac{k}{N_0}(s_B-s_A)]\to M$ given by
the unique minimal geodesic in $(M,g)$ from $G(A,B)(s_A+\frac{k-1}{N_0}(s_B-s_A))$
to $G(A,B)(s_A+\frac{k}{N_0}(s_B-s_A))$, $k=1,\ldots,N_0+1$. Moreover we set
\[
\gamma_{A,B}(s) =\hat G(A,B)(s) \text{ if }s \in [0,s_A] \cup [s_B,1].
\]

Since the minimal geodesic in any convex normal neighborhood depend
continuously (with respect to the $C^2$--norm) on its endpoints, $\gamma_{A,B}$ depends continuously by $(A,B)$ in the $H^{1}$--norm.
Moreover thanks to \eqref{eq:corrispondenza}, $\gamma_{A,B}$ satisfies \eqref{corde1}--\eqref{corde3} provided that $N_0$ is sufficiently large.

Now, using the flow $\eta^{-}$ of \eqref{eq:flusso-}, defined also in a neighborhood of $\phi^{-1}([-\delta_0,\delta_0])$, we can modified
$\gamma_{A,B}$ in $[s_A,s_B]$ obtaining $G$ such that $\phi(G(A,B)(s)) > -\delta_0$ for any $s \in ]s_A,s_B[$. Then, thanks to
\eqref{eq:corrispondenza} $G$ satisfies also
property \eqref{corde5}.
\end{proof}

We set
\begin{equation}\label{eq:2.6bis}
\begin{split}
&\mathfrak C=\big\{G(A,B)\,:\,A,B\in\partial\Omega\big\},\\
&\mathfrak C_0=\{G(A,A)\,:\,A\in\partial\Omega\}.\end{split}\end{equation}

\begin{rem}\label{rem:conseguenze-homeo}
Note  that $\mathfrak C$ is homeomorphic to $\mathbb{S}^{m-1}\times \mathbb{S}^{m-1}$ by a homeomorphism
mapping $\mathfrak C_0$ onto $\{(A,A)\,:\,A\in \mathbb{S}^{m-1}\}$.
\end{rem}

\medskip

Define now the following constant:
\begin{equation}\label{eq:defM0}
M_0=\sup_{x\in\mathfrak C}\int_0^1g(\dot x,\dot x)\,\mathrm dt.
\end{equation}
Since $\mathfrak C$ is compact and the integral in
\eqref{eq:defM0} is continuous in the $H^1$-topology, then
$M_0<+\infty$.
\bigskip
Finally we define the following subset of $\mathfrak M_0$:
\begin{equation}\label{eq:defMM}
\mathfrak M=\Big\{x\in\mathfrak M_0: \frac12\int_a^bg(\dot x,\dot
x)\,\mathrm dt< M_0 \quad \forall [a,b] \in \mathcal I_x \Big\}.
\end{equation}

We shall work in $\mathfrak M$ using flows in
$H^1\big([0,1],\R^m\big)$ for which $\mathfrak M$ is invariant.


\section{Geometrically critical values and variationally
critical portions}\label{sec:functional}

In this section we will introduce two different notions
of {\em criticality\/} for curves in $\Mcal$.

\begin{defin}\label{thm:defgeomcrit}
A number $c\in\left]0,M_0\right[$ will be called a {\em geometrically critical
value} if there exists an OGC $\gamma$ parameterized in $[0,1]$ such that
$\frac12\int_0^1g(\dot\gamma,\dot\gamma)\,\text dt=c$.
A number which is not geometrically critical will be called  {\em geometrically regular value}.
\end{defin}

\bigskip

It is important to observe that, in view to obtain multiplicity results,
distinct geometrically critical values yield
geometrically distinct orthogonal geodesic chords:
\begin{prop}\label{thm:distinct}
Let $c_1\ne c_2$, $c_1,c_2>0$ be distinct geometrically critical
values  with corresponding OGC $x_1,x_2$. Then
$x_1\big([0,1]\big)\ne x_2\big([0,1]\big)$.
\end{prop}
\begin{proof}
The OGC's $x_1$ and $x_2$ are parameterized in the interval $[0,1]$. Assume by contradiction,
$x_1([0,1])=x_2([0,1])$. Since
\[
x_{i}(]0,1[) \subset \Omega \text{ for any }i=1,2,
\]
we have
\[\{x_1(0),x_1(1)\}=
\{x_2(0),x_2(1)\}.\]
Up to reversing the orientation of $x_2$,  we can assume $x_1(0)=x_2(0)$.
Since $x_1$ and $x_2$ are OGC's, $\dot x_1(0)$ and $\dot x_2(0)$
are parallel, but the condition $c_1\ne c_2$ says that $\dot x_1(0)\ne\dot x_2(0)$.
Then  there exists $\lambda > 0, \lambda\not=1$ such that $\dot x_2(0) =\lambda \dot x_1(0)$ and therefore,
by the uniqueness of the Cauchy problem for geodesics we have $x_2(s)=x_1(\lambda s)$. Up to exchange $x_1$ with $x_2$
we can assume $\lambda > 1$. Since $x_2(\frac{1}{\lambda}) = x_1(1) \in\partial\Omega$, the transversality of
$\dot x_2(0)$ to $\partial\Omega$ implies the existence of $\bar s \in ]\frac{1}{\lambda},1]$ such that
$x_2(\bar s) \not\in\overline\Omega$, getting a contradiction.
\end{proof}

A notion of criticality will now be given in terms of
variational vector fields.
For $x\in \Mcal$, let $\mathcal V^+(x)$ denote
the following closed convex cone of $T_x H^1\big([0,1],\R^m\big)$:
\begin{equation}\label{eq:defV+(x)}
\mathcal V^+(x)=\big\{V\in T_x H^1\big([0,1],\R^m\big):g\big(V(s),
\nabla\phi\big(x(s)\big)\big)\ge0\ \text{for $x(s)\in\partial\Omega$}\big\};
\end{equation}
vector fields in $\mathcal V^+(x)$ are interpreted as
infinitesimal variations of $x$ by curves stretching ``outwards''
from the set $\overline\Omega$.

\begin{defin}\label{thm:defvariatcrit}
Let $x\in\mathfrak M$ and $[a,b] \subset [0,1]$; we say that
$x\vert_{[a,b]}$ is a $\mathcal V^{+}$--\emph{\em variationally critical portion\/} of
$x$ if $x\vert_{[a,b]}$ is not constant and if
\begin{equation}
\label{eq:varcriticality-plus}\int_a^bg\big(\dot x,\Ddt V\big)\,\mathrm
dt\ge0, \quad\forall\,V\in\mathcal V^+(x).
\end{equation}
\end{defin}

Similarly, for $x\in\Mcal$ we
define the cone:
\begin{equation}\label{eq:defV-(x)}
\mathcal V^-(x)=\big\{V\in T_x H^1\big([0,1],\R^m\big):g\big(V(s),
\nabla\phi\big(x(s)\big)\big)\le0\ \text{for $x(s)\in
\partial \Omega$}\big\},
\end{equation}
and we give the following

\begin{defin}\label{thm:defvariatcrit-}
Let $x\in\mathfrak M$ and $[a,b] \subset [0,1]$; we say that
$x\vert_{[a,b]}$ is a $\mathcal V^{-}$--\emph{variationally critical portion\/} of
$x$ if $x\vert_{[a,b]}$ is not constant and if
\begin{equation}
\label{eq:varcriticality}\int_a^bg\big(\dot x,\Ddt V\big)\,\mathrm
dt\ge0, \quad\forall\,V\in\mathcal V^-(x).
\end{equation}
\end{defin}

The integral in \eqref{eq:varcriticality} gives precisely the
first variation of the geodesic action functional in $(M,g)$ along
$x\vert_{[a,b]}$. Hence, variationally critical portions are
interpreted as those curves $x\vert_{[a,b]}$ whose geodesic energy
is {\em not decreased\/} after  infinitesimal variations by curves
stretching outwards from the set $\overline\Omega$. The motivation
for using outwards pushing infinitesimal variations is due to the
concavity of $\overline\Omega$. Indeed in the convex case it is customary
to use a curve shortening method in $\overline \Omega$,
that can be seen as the use of a flow constructed by
infinitesimal variations of $x$ in $\mathcal V^-(x)$,
keeping the endpoints of $x$ on $\partial\Omega$.

Flows obtained as integral flows of convex combinations of vector fields in $\mathcal V^+(x)$ play, in a certain sense,
the leading role in our variational approach. However we shall use also integral flows
of convex combinations of vector fields in $\mathcal V^-(x)$ to avoid certain variationally critical portions that do not correspond
to OGC's.
\smallskip

Clearly, we are interested in determining existence of
geometrically critical values. In order to use a variational
approach we will first have to keep into consideration the more general
class of $\mathcal V^{+}$--variationally critical portions. A central issue in our
theory consists in studying the relations between $\mathcal V^{+}$--variationally
critical portions $x\vert_{[a,b]}$ and OGC's. From now on
$\mathcal V^{+}$--variationally critical portions, will be called simply
variationally critical portions.

\section{Classification of variationally critical portions}
\label{sub:description}
Let us now take a look at how variationally critical portions look like. In first place, let us point out
that regular variationally critical portions are OGC's. In order to prove this, the following Lemma is crucial. Its proof can be found in \cite{esistenza}.

\begin{lem}\label{thm:leminsez4.4}
Let $x\in\mathfrak M$ be fixed, and let $[a,b]\in[0,1]$ be such
that $x\vert_{[a,b]}$ is a (non--constant) variationally critical
portion of $x$, with $x(a),x(b)\in\partial\Omega$ and
$x\big([a,b]\big)\subset\overline\Omega$. Then:
\begin{enumerate}
\item\label{itm:numfinint} $x^{-1}\big(\partial\Omega)\cap[a,b]$ consists of
a finite number of closed intervals and isolated points;
\item\label{itm:constconcomp} $x$ is constant on each connected
component of $x^{-1}\big(\partial\Omega)\cap[a,b]$;
\item\label{itm:piecC2} $x\vert_{[a,b]}$ is piecewise $C^2$, and the discontinuities of
$\dot x$ may occur only at points in $\partial\Omega$;
\item\label{itm:portgeo} each $C^2$ portion of $x\vert_{[a,b]}$ is a geodesic
in $\overline\Omega$.
\item\label{itm:addizionale} $\inf\{\phi(x(s))\,:\,s\in[a,b]\}<-\delta_0$.
\end{enumerate}
\end{lem}

Using the previous Lemmas, we can now prove the following:
\begin{prop}\label{thm:regcritpt}
Assume that there are no WOGC's  in $\overline\Omega$. Let
$x\in\Mcal$ and $[a,b]\in {\mathcal I_x^0}$ be such that
$x\vert_{[a,b]}$ is a variationally critical portion of $x$ and such
that the restriction of $x$ to $[a,b]$ is of class $C^1$. Then,
$x\vert_{[a,b]}$ is an orthogonal geodesic chord in
$\overline\Omega$.
\end{prop}
\begin{proof}
$C^1$--regularity, together with \eqref{itm:numfinint} and \eqref{itm:constconcomp} of Lemma \ref{thm:leminsez4.4}, show that
$x^{-1}(\partial\Omega)\cap[a,b]$ consists only of a finite number of isolated points.
Then, by the $C^1$ regularity on $[a,b]$ and parts \eqref{itm:piecC2}--\eqref{itm:portgeo} of Lemma \ref{thm:leminsez4.4}, $x$ is
a geodesic on the whole interval $[a,b]$. Moreover an integration by parts argument shows
that $\dot x(a)$ and $\dot x(b)$ are orthogonal to $T_{x(a)}\partial\Omega$ and
$T_{x(b)}\partial\Omega$ respectively. Finally, since there are no WOGC's on $\overline\Omega$, $x\vert_{[a,b]}$ is an OGC.
\end{proof}

Variationally critical portions $x\vert_{[a,b]}$ of class $C^1$ will be called
{\em regular variationally critical portions};
those critical portions that do not belong to this class will
be called {\em irregular}.
Irregular variationally critical portions of curves $x\in\Mcal$
are further divided into two subclasses, described in the Proposition below, whose proof can be obtained using Lemma \ref{thm:leminsez4.4} as done for the proof
of Proposition \ref{thm:regcritpt}.
\begin{prop}\label{thm:irregcritpt}
Assume that there are not WOGC's in $\overline\Omega$.
Let $x\in\Mcal$ and let $[a,b]\in {\mathcal I_x^0}$ be such that
$x\vert_{[a,b]}$ is an irregular variationally critical portion of
$x$. Then, there exists a subinterval $[\alpha,\beta]\subset
[a,b]$ such that $x\vert_{[a,\alpha]}$ and $x\vert_{[\beta,b]}$
are constant (in $\partial\Omega$), $\dot x(\alpha^+)\in
T_{x(\alpha)}(\partial\Omega)^\perp$, $\dot x(\beta^-)\in
T_{x(\beta)}(\partial\Omega)^\perp$, and one of the two mutually
exclusive situations occurs:
\begin{enumerate}
\item \label{itm:cusp} there exists a finite number of intervals
$[t_1,t_2]\subset\left]\alpha,\beta\right[$ such that
$x\big([t_1,t_2]\big)\subset\partial\Omega$ and that are maximal
with respect to this property; moreover, $x$ is constant on each
such interval $[t_1,t_2]$, and $\dot x(t_1^-)\ne\dot x(t_2^+)$;
\item \label{itm:stop} $x\vert_{[\alpha,\beta]}$ is an OGC in
$\overline\Omega$.
\end{enumerate}
\end{prop}

Irregular variationally critical portions in the class described
in part  \eqref{itm:cusp} will be called {\em of first type},
those described in part \eqref{itm:stop} will be called {\em of
second type}. An interval $[t_1,t_2]$ as in part \eqref{itm:cusp}
will be called a {\em cusp interval\/} of the irregular critical
portion $x$.

\begin{rem}\label{rem:rem4.10}
We observe here that, due to the strong concavity assumption, if
$x\in\Mcal$ is an irregular variationally critical point of
first  type and $[t_1,t_2],[s_1,s_2]$ are cusp intervals for $x$
contained in $[a,b]$ with $t_2<s_1$,    then
\[\text{there exists}\  s_0\in\left]t_2,
s_1\right[\text{\ with\ }\phi(x(s_0))< -\delta_0,\] (see Remark\ref{thm:remopencondition}). This implies
that the number of cusp intervals of irregular variationally
critical portions $x\vert_{[a,b]}$,  is uniformly bounded (see Lemma \ref{thm:lemmacazzatina}).

We also remark that at each cusp interval $[t_1,t_2]$ of $x$, the vectors $\dot
x(t_1^-)$ and $\dot x(t_2^+)$ may not be  orthogonal to
$\partial\Omega$. If $x\vert_{[a,b]}$ is a irregular critical
portion of the first type, and if $[t_1,t_2]$ is a cusp interval of
for $x$, we will set
\begin{equation}\label{eq:4.15}
\Theta_x(t_1,t_2) = \text{the (unoriented) angle between the
vectors\ } \dot x(t_1^-) \text{\ and\ } \dot x(t_2^+);
\end{equation}
observe that $\Theta_x(t_1,t_2)\in\left]0,\pi\right]$.
\end{rem}

\begin{rem}\label{thm:rem4.11bis}We observe that if $[t_1,t_2]$ is
a cusp interval for $x$, then the tangential components of $\dot
x(t_1^-)$ and of $\dot x(t_2^+)$ along $\partial\Omega$ are equal;
this is is easily obtained with an integration by parts argument. It
follows that if $\Theta_x(t_1,t_2)>0$, then $\dot x(t_1^-)$ and
$\dot x(t_2^+)$ cannot be both tangent to $\partial\Omega$.
\end{rem}

We will denote by $\mathcal Z$ the set of all curves having variationally critical portions:
\[\mathcal Z=\big\{x\in\Mcal:\exists\,[a,b]\subset[0,1]\ \text{such that
$x\vert_{[a,b]}$ is a variationally critical portion of $x$}
\big\};\] the following compactness property holds for $\mathcal
Z$:
\begin{prop}\label{thm:Zrcompatto}
If $x_n$ is a sequence in $\mathcal Z$ and $[a_n,b_n]\in{\mathcal
J_{x_n}^0}$ is  such that ${x_n}\vert_{[a_n,b_n]}$ is a (non-constant)
variationally critical portion of $x_n$, then, up to subsequences,  as $n\to\infty$
$a_n$ converges to some $a$, $b_n$ converges to some $b$, with
$0\le a<b\le1$, and the sequence of paths $x_n:[a_n,b_n]\to\overline\Omega$ is $H^1$-convergent (in the sense of Remark \ref{rem:rep}) to some
curve $x:[a,b]\to\overline \Omega$ which is variationally critical.
\end{prop}
\begin{proof}
By Lemma~\ref{thm:lemmacazzatina}, $b_n-a_n$
is bounded away from
$0$, which implies the existence of subsequences converging in
$[0,1]$ to $a$ and $b$ respectively, and with $a<b$. If $x_n$ is a
sequence of regular variationally critical portions, then the
conclusion follows easily observing that $x_n$, and thus $\widehat
x_n$ (its affine reparameterization in $[a,b]$) is a sequence of
geodesics with image in a compact set and having  bounded energy.

For the general case, one simply observes that the number of cusp
intervals of each $x_n$ is bounded uniformly in $n$, and the
argument above can be repeated by considering the restrictions of
$x_n$ to the complement of the union of all cusp intervals.
Finally, using partial integration of the term $\int_a^b g(\dot x,\Ddt
V)\,\text dt$, one observes that it is nonnegative for all $
V\in\mathcal V^+(x)$, hence $x$ is variationally critical.
\end{proof}

\begin{rem}\label{rem:rem4.12}
We point out  that the first part of the proof of Proposition
\ref{thm:Zrcompatto} shows that if $x_n\in\mathcal Z$ and
$[a_n,b_n]\in {\mathcal I_{x_n}^0}$ is an interval such that
${x_n}\vert_{[a_n,b_n]}$ is an OGC, then, up to subsequences, there
exists $[a,b]\subset[0,1]$ and $x:[a,b]\to\overline\Omega$ such
that ${x_n}\vert_{[a_n,b_n]}\to x\vert_{[a,b]}$ in $H^1$ and
$x$ is an OGC.
\end{rem}

Since we are assuming that there are no  WOGC in $\overline\Omega$,
by  Lemma~\ref{thm:leminsez4.4}, Proposition~\ref{thm:regcritpt},
Proposition~\ref{thm:irregcritpt} and Proposition~\ref{thm:Zrcompatto},
we obtain immediately the following result.

\begin{cor}\label{thm:cor4.11bis}There exists $d_0>0$ such that for any
$x\vert_{[a,b]}$ irregular variationally portion of first type with  $[a,b] \in \mathcal I_x^0$, there exists
a cusp interval $[t_1,t_2]\subset[a,b]$ for $x$ such that
\[
\Theta_x(t_1,t_2)\ge d_0.
\]
\end{cor}

\section{The notion of topological non-essential interval}\label{V-}

As observed in \cite{esistenza}, we need three different types of flows, whose formal definition
will be given below. ``Outgoing flows'' are applied to paths that are \emph{far} from variationally critical portions
(cf.\ Definition \ref{thm:defvariatcrit}). ``Reparameterization flows'' are applied to curves that are \emph{close}  to irregular variational portions of second type. ``Ingoing flows'' are used to avoid irregular variational portions of first type. In order to describe this type of homotopies, we introduce the notion of
\textit{topological non-essential interval}, which is a key point in defining the admissible homotopies. The possibility of avoiding irregular variational portions of first type is based on the following {\em regularity\/} property
of the critical variational portions with respect to ingoing directions.

\begin{lem}\label{thm:lem5.3}
Let $y\in H^1\big([a,b],\overline\Omega\big)$ be such that:
\begin{equation}\label{eq:eq5.6}
\int_a^bg\big(\dot y,\Ddt V)\,\mathrm dt\ge0,\qquad\forall\,V\in\mathcal V^-(y)
\ \text{with}\ V(a)=V(b)=0.
\end{equation}
Then, $y\in H^{2,\infty}([a,b],\overline\Omega)$ and in particular it is of class $C^1$.
\end{lem}
\begin{proof}
See for instance \cite[Lemma 3.2]{London}.
\end{proof}
\begin{rem}\label{rem:4.16bis}
Note that, under the assumption of strong concavity, the set \[C_y=\big\{s\in[a,b]\,:\,\phi(y(s))=0\big\}\] consists
of a finite number of intervals. On each one of these intervals, $y$ is of class $C^2$, and it satisfies the ``constrained geodesic'' differential equation
\begin{equation}\label{eq:4.30bis}
\Dds\dot y(s)=-\left[\frac1{g(\nu(y(s)),\nabla\phi(y(s)))}H^\phi(y(s))[\dot y(s),\dot y(s)]\right]\nu(y(s)).
\end{equation}
\end{rem}

\begin{rem}\label{rem:4.17}
For every $\delta\in\left]0,\delta_0\right]$ we have the following property: for any $x\in\mathfrak M$ and $[a,b] \in {\mathcal I}_x$
such that $x\vert_{[a,b]}$ is an irregular variationally critical
portion of first type, there exists an interval $[\alpha,\beta]\subset[a,b]$ and
a cusp interval  $[t_1,t_2]\subset[\alpha,\beta]$ such that:
\begin{equation}\label{eq:alphabeta}
\Theta_x(t_1,t_2)\ge d_0, \text{ and } \phi(x(\alpha))=\phi(x(\beta))=-\delta,
\end{equation}
where $d_0$ is given in Corollary \ref{thm:cor4.11bis}.

Note that $g\big(\nabla\phi(x(\alpha)),\dot x(\alpha)\big)>0$ and  $g\big(\nabla\phi(x(\beta)),\dot x(\beta)\big)<0$,
by the strong concavity assumption.
\end{rem}
\medskip

For the remaining of the paper we will
denote by \[\pi:\phi^{-1}\big([-\delta_0,0]\big)\longrightarrow\phi^{-1}(0)\] the
retraction onto $\partial\Omega$ obtained from the inverse of the exponential map
of the normal bundle of $\phi^{-1}(0)$. By Remark \ref{thm:rem4.11bis},
a simple contradiction argument
shows that the following properties are satisfied by irregular variationally critical portions of first type
(see also Corollary \ref{thm:cor4.11bis}):

\begin{lem}\label{thm:lem4.18}
There exists $\bar\gamma>0$ and $\delta_1\in\left]0,\delta_0\right[$ such that, for all $\delta\in\left]0,\delta_1\right]$, for any $x\in{\mathfrak M}$
such that $x\vert_{[a,b]}$ is an irregular variationally critical portion of first type,
and for any interval $[\alpha,\beta] \subset [a,b]$ that contains a cusp interval
$[t_1,t_2]$ satisfying \eqref{eq:alphabeta},
the following inequality holds:
\begin{equation}\label{eq:4.31}
\max\Big\{\Vert x(\beta)-\pi(x(\alpha))\Vert_{E},\,
\Vert x(\alpha)-\pi(x(\beta))\Vert_{E}\Big\}\ge(1+2\bar\gamma)
\Vert\pi(x(\beta))-\pi(x(\alpha))\Vert_{E},
\end{equation}
\end{lem}
\noindent (recall that $\Vert \cdot \Vert_{E}$ denotes the Euclidean norm).

\bigskip

The following Lemma says that curves satisfying \eqref{eq:4.31}
and those that  satisfy \eqref{eq:eq5.6} are contained in \emph{disjoint} closed subsets;
in other words, curves satisfying \eqref{eq:4.31} are far from
being critical with respect to $\mathcal V^-$. In particular,
the set of irregular variationally critical portions of first type consists of curves
at which the value of the energy functional can be decreased by deforming in the directions
of $\mathcal V^-$.

Let $\bar\gamma$ be as in Lemma~\ref{thm:lem4.18}.
\begin{lem}\label{thm:lem4.19} There exists $\delta_2\in\left]0,\delta_0\right[$ with the following property: for any $\delta\in\left]0,\delta_2\right]$,
for any $[a,b]\subset\R$ and for any $y\in H^1([a,b],\overline\Omega)$ satisfying \eqref{eq:eq5.6} and
\[
\phi(y(a))=\phi(y(b))=-\delta,
\phi(y(\bar t))=0 \text{ for some }\bar t\in ]a,b[,
\]
the following inequality holds:
\begin{equation}\label{eq:4.32}
\max\Big\{\Vert y(b)-\pi(y(a))\Vert_{E},\,\Vert y(a)-\pi(y(b))\Vert_{E}\Big\}
\le \left(1+\frac{\bar\gamma}2\right)\Vert\pi(y(b))-\pi(y(a))\Vert_{E}.
\end{equation}
\end{lem}

\begin{proof}
See \cite{esistenza}.
\end{proof}
Using vector fields in $\mathcal V^-(x),\,x\in\mathfrak M$, we can build a flow moving away from
the set of irregular variationally critical portions of first type,
without increasing the energy functional. To this aim let $\pi,\,\bar\gamma,\,\delta_1,\,\delta_2$ be chosen
as in Lemma \ref{thm:lem4.18} and \ref{thm:lem4.19}, and set
\begin{equation}\label{eq:4.37}
\bar\delta=\min\{\delta_1,\delta_2\}.
\end{equation}
Let us give the following:
\begin{defin}\label{thm:defcostapp}
Let $x\in\mathfrak M$, $[a,b]\in \mathcal I_x^{0}$ and $[\alpha,\beta]\subset[a,b]$.
We say that \emph{$x$ is $\bar\delta$-close to $\partial\Omega$ on $[\alpha,\beta]$} if the following situation occurs:
\begin{enumerate}
\item\label{itm:app1} $\phi(x(\alpha))=\phi(x(\beta))=-\bar\delta$;
\item\label{itm:app2} $\phi(x(s))\ge-\bar\delta$ for all $s\in[\alpha,\beta]$;
\item\label{itm:app3}  there exists $s_0\in\left]\alpha,\beta\right[$ such that $\phi(x(s_0))>-\bar\delta$;
\item $[\alpha,\beta]$ is minimal with respect to properties \eqref{itm:app1}, \eqref{itm:app2} and \eqref{itm:app3}.

\end{enumerate}
If $x$ is $\bar\delta$-close to $\partial\Omega$ on $[\alpha,\beta]$, the \emph{maximal proximity} of
$x$ to $\partial\Omega$ on $[\alpha,\beta]$ is defined to be the quantity
\begin{equation}\label{eq:maxprox}
\mathfrak p^x_{\alpha,\beta}=\max_{s\in[\alpha,\beta]}\phi(x(s)).\end{equation}
\end{defin}
Given an interval $[\alpha,\beta]$  where $x$ is $\bar\delta$-close to $\partial\Omega$, we define
the following constant, which is a sort of measure of how much the curve $x\vert_{[\alpha,\beta]}$
fails to flatten along $\partial\Omega$:
\begin{defin}\label{thm:defflat}
The \emph{bending constant} of $x$ on $[\alpha,\beta]$ is defined by:
\begin{equation}\label{eq:bendconst}
\mathfrak b^x_{\alpha,\beta}=
\frac{\max\big\{
\Vert x(\beta)-\pi(x(\alpha))\Vert_{E},\Vert x(\alpha)-\pi(x(\beta))\Vert_{E}\big\}}{\Vert \pi(x(\alpha))-\pi(x(\beta))\Vert_{E}}\in\R^+\cup\{+\infty\},
\end{equation}
where $\pi$ denotes the projection onto $\partial\Omega$ along orthogonal geodesics.
\end{defin}

We observe that $\mathfrak b^x_{\alpha,\beta}=+\infty$ if and only if $x(\alpha)=x(\beta)$.\smallskip

Let $\bar\gamma$ be as in Lemma~\ref{thm:lem4.18}.
If the bending constant of a path $y\vert_{[\alpha,\beta]}$
is greater than or equal to $1+\bar\gamma$, then the energy functional in the interval $[\alpha,\beta]$ can be decreased in a
neighborhood of $y\vert_{[\alpha,\beta]}$ keeping the endpoints $y(\alpha)$ and $y(\beta)$
fixed, and moving away from $\partial \Omega$ (cf. \cite{esistenza}).

In order to prove this, we first need the following
\begin{defin}\label{def:summary-interval}
An interval $[\tilde\alpha,\tilde\beta]$ is called a \emph{summary interval} for $x \in \mathfrak M$ if it is the smallest interval contained in
$[a,b]\in\mathcal I_{x}^{0}$ and containing all the intervals $[\alpha,\beta]$ such that
\begin{itemize}
\item $x$ is $\bar\delta$--close to $\partial\Omega$ on $[\alpha,\beta]$,
\item $b^{x}_{\alpha,\beta} \geq 1 + \bar \gamma$.
\end{itemize}
\end{defin}
The following result is proved in \cite{esistenza}:
\begin{prop}\label{thm:prop4.20}
There exist positive constants $\sigma_0\in\left]0,{\bar\delta}/2\right[$,
$\varepsilon_0 \in\left]0,\bar \delta - 2\sigma_0\right[$,
$\rho_0,\,\theta_0$ and $\mu_0$ such that for all
$y\in\mathfrak M$, for all $[a,b]\in \mathcal I_y$
and for all $[\tilde\alpha,\tilde\beta]$ summary interval for $y$ containing an interval $[\alpha,\beta]$ such that :
\[
y \text{ is $\bar\delta$--close to $\partial\Omega$ on }[\alpha,\beta], \,
\mathfrak b^y_{\alpha,\beta}\ge1+\bar\gamma, \, \mathfrak p^y_{\alpha,\beta}\ge-2\sigma_0,
\]
there exists $V_y\in H_0^1\big([\tilde\alpha,\tilde\beta],\R^m\big)$ with the following property:

for all $z\in H^1([\tilde\alpha,\tilde\beta],\R^m)$ with $\Vert z-y\Vert_{\tilde \alpha, \tilde \beta}\le\rho_0$ it is:
\begin{enumerate}
\item\label{itm:teruno} $V_y(s)=0$ for all $s\in[\tilde\alpha,\tilde\beta]$ such that $\phi(z(s))\le-\bar\delta + \varepsilon_0$;
\item\label{itm:terdue}
$g\big(\nabla\phi(z(s)),V_y(s)\big)\le-\theta_0\Vert V_y\Vert_{\tilde\alpha,\tilde\beta}$, if $s\in[\tilde\alpha,\tilde\beta]$
and $\phi(z(s))\in[-2\sigma_0,2\sigma_0]$
\item\label{itm:tertre} $\int_{\tilde\alpha}^{\tilde\beta} g(\dot z,\Ddt V_y)\,\mathrm dt\le-\mu_0\Vert V_y\Vert_{\tilde\alpha,\tilde\beta}.$
\end{enumerate}
\end{prop}

\begin{rem}\label{rem:sigma1}
As observed in \cite{esistenza}, in order to define flows that move away from curves having topologically non-essential intervals (defined below), it will be necessary to fix a constant $\sigma_1 \in\left]0,\sigma_0\right[$
such that
\[
\sigma_1 \leq \frac27\rho_0\theta_0,
\]
where $\rho_0, \, \theta_0$ are given by Proposition  \ref{thm:prop4.20}.
\end{rem}
\medskip

Proposition \ref{thm:prop4.20} and Remark \ref{rem:sigma1} are  crucial ingredients for the
 definition of the class of the admissible homotopies, whose elements will avoid
irregular variationally critical
points of first type. The description of this class is based on the notion of topologically non-essential interval given below.
\medskip

Let $\bar\delta$ be as in \eqref{eq:4.37}, $\bar\gamma$ as in Lemma \ref{thm:lem4.18}
and $\sigma_1$ as in Remark \ref{rem:sigma1}.
\begin{defin}\label{def:4.21}
Let $y\in \mathfrak M$ be fixed. An interval $[\alpha,\beta]\subset[a,b]\in{\mathcal I}_{y}$, is called
\emph{topologically not essential interval (for $y$)} if  $y$ is $\bar\delta$-close to $\partial\Omega$ on
$[\alpha,\beta]$, with $\mathfrak p^y_{\alpha,\beta}\ge-\sigma_1$ and
$\mathfrak b^y_{\alpha,\beta}\ge(1+\tfrac32\bar\gamma)$.
\end{defin}

\begin{rem}\label{rem:4.22}
By Lemma \ref{thm:lem4.18} the intervals $[\alpha,\beta]$  containing cusp intervals
$[t_1,t_2]$ of curves $x$, which are
irregular variationally critical portion of first type, and satisfying $\Theta_x(t_1,t_2)\ge d_0$ are topologically not essential intervals with
$\mathfrak p_{\alpha,\beta}^x=0$
and $\mathfrak b_{\alpha,\beta}^x\ge 1+2\bar\gamma$.
This fact will allow us to  move away from the set of irregular variationally critical portions of first type
without increasing the value of the energy functional.
\end{rem}

\section{The admissible homotopies}\label{sec:homotopies}

In the present section
we shall list the properties of the admissible homotopies used
in our minimax argument. The notion of topological critical level
used in this paper,
depends on the choice of the admissible homotopies.

We shall consider continuous homotopies
$h:[0,1]\times\Dcal\to\mathfrak M$ where
$\mathcal D$ is a closed subset of $\mathfrak C$. It should be observed, however, that
 totally analogous definitions apply also to any element $h$ in $\mathfrak M$,
 not necessarily contained in $\mathfrak C$.

Recall that $\mathfrak C$ is described in \eqref{eq:2.6bis}. First of all, we require that:

\begin{equation}\label{eq:numero1}
h(0,\cdot) \text{ is the inclusion of } \Dcal
\text{ in }\mathfrak M.
\end{equation}

The homotopies that we shall use are of three types: outgoing homotopies, reparameterizazions and ingoing homotopies.
They can be described in the following way.

\begin{defin}\label{thm:tipoA}
Let $0 \leq \tau' < \tau'' \leq 1$. We say that $h$ is of type $A$
in $[\tau',\tau'']$ if it satisfies the following property:
\begin{enumerate}
\item\label{unico}
for all $\tau_0 \in [\tau',\tau'']$, for all $s_0 \in [0,1]$, for all $x \in \mathcal D$, if
$\phi(h(\tau_0,x)(s_0) = 0$, then
$\tau \mapsto \phi(h(\tau,x)(s_0)$ is strictly increasing in a neighborhood of $\tau_0$.
\end{enumerate}
\end{defin}

\begin{rem}\label{rem:monotonia-intervalli}
It is relevant to observe that, by property above of
Definition \ref{thm:tipoA}, if $[a_{\tau},b_{\tau}]$ denotes any
interval in $\mathcal I_{h(\tau,\gamma)}$ we have:
\begin{equation*}\label{eq:4.35f}
\tau' \le\tau_1<\tau_2\le\tau''
\text{ and }
[a_{\tau_1},b_{\tau_1}]\cap
[a_{\tau_2},b_{\tau_2}]\ne\emptyset \Longrightarrow
[a_{\tau_2},b_{\tau_2}] \subset
[a_{\tau_1},b_{\tau_1}].
\end{equation*}
\end{rem}


In the next Definition we describe the admissible homotopies consisting in suitable reparameterizations $\Lambda(\tau,\gamma)$. The deformation parameter $\tau$ moves in a fixed interval $[\tau',\tau'']$.

\begin{defin}\label{thm:tipoB}
Let $0 \leq \tau' < \tau'' \leq 1$. We say that $h$ is of type $B$
in $[\tau',\tau'']$ if it satisfies the following property:
there exists $\Lambda : [\tau',\tau''] \times \mathcal H_0^1([0,1],[0,1]) \rightarrow [0,1]$ continuous and such that
\begin{itemize}
\item $\Lambda(\tau,\gamma)(0)=0,\,\Lambda(\tau,\gamma)(1)=1, \, \forall \tau \in [\tau',\tau''],\,\forall \gamma \in \mathcal D$;

\item 
$s\mapsto\Lambda(\tau,\gamma)(s) \text{is strictly increasing in }[0,1], \; \forall \tau \in [\tau',\tau''],\forall \gamma \in \mathcal D$;

\item
$\Lambda(0,\gamma)(s) = s \text{ for any }\gamma \in \mathcal D, s \in [0,1]$;

\item
$h(\tau,\gamma)(s) = (\gamma \circ \Lambda(\tau,\gamma))(s) \;
\forall \tau \in [\tau',\tau''], \forall s \in [0,1], \forall \gamma
\in \mathcal D$.
\end{itemize}

\end{defin}

\begin{defin}\label{thm:tipoC}
Let $0 \leq \tau' < \tau'' \leq 1$. We say that $h$ is of type $C$
in $[\tau',\tau'']$ if it satisfies the following properties:
\begin{enumerate}
\item\label{eq:CI}
$h(\tau',\gamma)(s) \not \in \Omega \Rightarrow h(\tau,\gamma)(s) =
h(\tau',\gamma)(s)$ for any $\tau \in [\tau',\tau'']$;
\item\label{eq:CIbis}
$h(\tau',\gamma)(s) \in \Omega
\Rightarrow h(\tau,\gamma)(s) \in \Omega$
for any $\tau \in [\tau',\tau'']$;
\end{enumerate}
\end{defin}
\smallskip

The interval $[0,1]$ where $\tau$ varies will be partitioned in the
following way:
\begin{multline}\label{eq:partizione}
\text{There exists a partition of the interval } [0,1],\; 0=\tau_{0} < \tau_1 < \ldots <\tau_k =1 \text{ such that}\\
\text{ on any interval } [\tau_i, \tau_{i+1}], i=0,\ldots,k-1,
\text{ the homotopy $h$ is either of type A, or B, or C.}
\end{multline}

Homotopies of type A will be used away from variationally critical portions, homotopies of type B near variationally critical portions
of II type, while homotopies of type C will be used near variationally critical portions
of I type.\smallskip

Now, in order to move far from topologically non-essential
intervals (cf.\ Definition \ref{def:4.21}) we need the
following further property:

\begin{multline}\label{eq:numero8}
\text{if }[a,b]\in {\mathcal I}_{h(\tau,\gamma)}
\text{ then for all }
[\alpha,\beta]\subset[a,b] \text{ topologically non-essential it is }\\
\phi(h(\tau,\gamma)(s))\le-\frac{\sigma_1}2 \text {for all }s
\in[\alpha,\beta],
\end{multline}
where ${\sigma_1}$ is defined in Remark \ref{rem:sigma1}.
\bigskip

We finally define the following class of admissible homotopies:
\begin{multline}\label{eq:class0}
\mathcal H =\big \{(\mathcal D, h): \mathcal D \text{  is a closed subset of $\mathfrak C$ and } \\
h:[0,1] \times \mathcal D \rightarrow \mathfrak M \text{ satisfies
\eqref{eq:numero1}, \eqref{eq:partizione} and \eqref{eq:numero8}}\big\}.
\end{multline}

\begin{rem}\label{rem:6.6}
Obviously, it is crucial to have $\mathcal H \not= \emptyset$. But thanks to Lemma
\ref{thm:corde} we see that any $G(A,B)$ does not have topological non-essential intervals, and  denoting by
$I_{\mathfrak C}$ the constant identity homotopy we have $(\mathfrak C,I_{\mathfrak C}) \in
\mathcal H$.
\end{rem}

In order to introduce the functional for our minimax argument, we set
for any $(\mathcal D,h) \in \mathcal H$,
\begin{equation}\label{eq:funzionepreparatoria}
{\mathcal F}(\mathcal D,h) =
\sup\Big\{\tfrac{b-a}2\int_{a}^{b}g(\dot y,\dot y)\,\mathrm ds:
y=h(1,x), x\in{\mathcal D}, [a,b] \in {\mathcal I}_y\Big\}.
\end{equation}

\begin{rem}\label{rem:7.0}
It is interesting to observe that the integral $\tfrac{(b-a)}2\int_a^bg(\dot y,\dot y)\,\text dt$ coincides with  $\tfrac12\int_0^1g(\dot y_{a,b},\dot y_{a,b})\,
\mathrm dt$, where $y_{a,b}$ is the affine reparameterization of $y$ on the interval $[0,1]$.
\end{rem}

\begin{rem} Note also that, by the definition of $\mathcal H$, we have
\begin{equation}\label{eq:7.0bis}
\mathcal F(\mathcal D,h) < \frac{M_0}{2},\qquad\forall (\mathcal D,h)\in\mathcal H.
\end{equation}
\end{rem}

\bigskip

Given continuous maps $h_1:[0,1]\times F_1\to\mathfrak M$ and
$h_2:[0,1]\times F_2\to\mathfrak M$ such that $h_1(1,F_1)\subset
F_2$, then we define the \emph{concatenation} of $h_1$ and $h_2$ as
the continuous map $h_2\star h_1:[0,1]\times F_1\to\Lambda$ given by
\begin{equation}\label{eq:defconcatenationhomotop}
h_2\star h_1(t,x)=
\begin{cases}
h_1(2t,x),&\text{if\ } t\in[0,\tfrac12],\\
h_2(2t-1,h_1(1,x)),&\text{if\ } t\in[\tfrac12,1].
\end{cases}
\end{equation}
\section{Deformation Lemmas}\label{sec:second}

The first deformation result that we use is the analogous of the first (classical) deformation Lemma. By the same proof in \cite{esistenza} (without any use of symmetries properties on the flows) we obtain:

\begin{prop}[First Deformation Lemma]
\label{thm:firstdeflemma} Let $c \in\left ]0,M_0\right[$  be a geometrically  regular
value (cf. Definition \ref{thm:defgeomcrit}). Then, $c$ is a topologically regular value of $\mathcal F$, namely
there exists $\varepsilon=\varepsilon(c)>0$ with the following property:
for all  $(\mathcal D,h)\in {\mathcal H}$ with
\[
\mathcal F(\mathcal D,h)\leq c+\varepsilon
\]
there exists a
continuous map $\eta\in C^0\big([0,1]\times h(1,\Dcal),\mathfrak M\big)$ such that $(\Dcal,\eta\star h)\in{\Hcal}$ and
\[
\mathcal F(\Dcal,\eta\star h)\leq c-\varepsilon.
\]
\end{prop}

\begin{rem}
Let us recall here briefly the main idea behind the proof of Proposition~\ref{thm:firstdeflemma}, which is discussed in details in reference~\cite{esistenza}.
As in the classical Deformation Lemma, if $c$ is a regular value, then one shows there exist $\varepsilon > 0$ and a flow carrying the sublevel $c+\varepsilon$ inside the sublevel $c-\varepsilon$. The technical issue here is the fact  that we need flows
under which pieces of curves which are outside $\overline \Omega$ remain outside of $\overline \Omega$.
This is obtained as follows.
Using suitable pseudo-gradient vector fields, we first move away form curves having topologically non-essential intervals. Near irregular variational critical portions of second type, the desired flow is obtained by using reparameterizations, as described in Definition \ref{thm:tipoB}. Finally, we use flows described in Definition \ref{thm:tipoA} in order to move outside $\overline \Omega$ when we are far form variational critical portions of any type. A suitable partition of unity argument, needed to combine these different flows, allows to define the required homotopy that carries
the sublevel $c+\varepsilon$ into  the sublevel $c-\varepsilon$ if there are no OGC's having energy $c$.
\end{rem}

\medskip

We shall find positive geometrical critical level using the following Lemma, which is a simple consequence of Lemma \ref{thm:lemmacazzatina}.
\begin{lem}\label{lem:topological}
Suppose that
\[
\mathcal F(\mathcal D,h) < \frac12\left(\frac{3\delta_0}{4K_0}\right)^2.
\]
Then there exists an homotpy $\eta$ such that $(\eta \star h)(1,\gamma)(s)\in \partial \Omega$ for all $\gamma \in \mathcal D$, for any $s \in [0,1]$.
\end{lem}

In order to obtain an analogue of the classical Second Deformation Lemma, we first need to describe neighborhoods of critical curves that must be removed in order to make the functional $\mathcal F$ decrease.
We shall assume that the number of OGC's is finite; obviously such an assumption is not restrictive.

For every
$[a,b] \subset [0,1]$ and $\omega$ OGC parameterized in the interval $[0,1]$, we denote by $\omega_{a,b}$ the OGC $\omega$ affinely
reparameterized on the interval $[a,b]$. We shall consider only intervals $[a,b]$ such that
\begin{equation}\label{eq:limitazione-ab}
\int_a^bg(\dot \omega_{a,b},\dot \omega_{a,b})\mathrm ds \leq M_0.
\end{equation}
\bigskip

Since we are assuming that the number of OGC's is finite we can choose a positive $r_*$ sufficiently small so that
\begin{multline}\label{eq:rstar1}
\Vert \omega^{1}_{a,b} - \omega^{2}_{a,b}\Vert_{a,b} > 2r_{*}, \text{ for any } [a,b] \subset [0,1] \text{ satisfying \eqref{eq:limitazione-ab}}, \\
\text{ for any
$\omega^{1},\omega^{2}$ OGC's parameterized in $[0,1]$}
\text{ and such that $\omega^{1} \neq \omega^{2}$.}
\end{multline}

Note that \eqref{eq:rstar1} holds even if $\omega^{2}(s) =
\omega^{1}(1-s)$, because  $\omega_1$ is not constant.
Moreover, since for any OGC $\omega$ it is $\omega(0)\ne\omega(1)$ (by uniqueness in the geodesic Cauchy problem),
if $r_*$ is sufficiently small we have:
\begin{multline}\label{eq:rstar2}
\text{for any OGC $\omega$ parameterized in $[0,1]$,} \\
\{y \in \partial \Omega: \dist_{E}(y,\omega(0))\leq r_* \} \cap  \{y \in \partial \Omega: \dist_{E}(y,\omega(1))\leq r_* \} = \emptyset.
\end{multline}
(Recall that $\dist_{E}$ denotes the Euclidean distance in $\mathbb{R}^m$.)

Also note that $r_*$ can be chosen so small that
 
\begin{multline}\label{eq:rstar3}
\text{for any OGC
$\omega$,  the sets} \\
\text{ $\{\pi
(y)\,:\,\dist_{E}(y,\omega(0))<2r_*\}$ and
$\{\pi(y)\,:\,\dist_{E}(y,\omega(1))<2r_*\}$} \\
\text{are contractible in $\partial\Omega$},
\end{multline}
where $\pi:\phi^{-1}\big([-\delta_0,0]\big)\longrightarrow\phi^{-1}(0)$ is the
retraction onto $\partial\Omega$ obtained by the gradient flow for $\phi$.
\bigskip

For any $(\mathcal D,h) \in \mathcal H$,
and  $\omega$ orthogonal geodesic chord parameterized in $[0,1]$,
we set, for any $r \in ]0,r_*]$,
\begin{multline}\label{eq:9.1}
\Ucal(\mathcal D,h,\omega,r)=\big\{x = h(1,y): y \in \mathcal D
\text{ and there exists }[a,b] \in \mathcal I_x \text{ such that }\\
\Vert x\vert_{[a,b]} - \omega_{a,b} \Vert_{a,b} \leq r \big\},
\end{multline}

If $[a,b]$ satisfies the above property we say that $x_{[a,b]}$ is $r_*$--close to $\omega_{a,b}$.
Note that $\Ucal(\mathcal D,h,\omega,r)$ is closed in $\mathfrak M$ and we have
\begin{multline}\label{eq:9.2}
{\Ucal(\mathcal D,h,\omega_1,r_*)}\cap {\Ucal(\mathcal D,h,\omega_2,r_*)}=\emptyset,\quad \forall\,(\mathcal D,h)\in\mathcal H,\\
\forall\,\omega_1,\omega_2\text{\
OGC's parameterized in $[0,1]$}
\text{ and such that }
\omega_1 \neq \omega_2.
\end{multline}
Now  if $c > 0$ is a geometrically critical we set
\[
E_c = \{\omega \text{ OGC}: \int_0^1 g(\dot \omega,\dot \omega)ds = c \}
\]
and, for any $r \in ]0,r_*]$
\[
\Ucal_{r}(\mathcal D,h,c)=\bigcup_{\omega \in E_c} \Ucal(\mathcal D, h,\omega,r).
\]

\begin{rem}\label{thm:chiusura}
Fix $\varepsilon > 0$ so that $c-\epsilon > 0$ and consider
\begin{multline}\label{eq:Ac}
{\mathcal A}_{c,\varepsilon} = \{y \in \mathcal D: x = h(1,y) \in \Ucal_{r_*}(\mathcal D,h,c), \text{ and there exists $[a,b] \in \mathcal I_x$ such that }\\
x\vert_{[a,b]} \text{ is $r_*$--close to }\omega_{a,b} \text{ and }\frac{b-a}2 \int_a^bg(\dot x,\dot x)\,\mathrm ds \in [c-\varepsilon,c+\varepsilon]\}.
\end{multline}
\end{rem}

Again, by the same proof in \cite{esistenza}, we obtain the following
\begin{prop}[Second Deformation Lemma]\label{thm:prop9.3} Let $c \geq \frac12\big(\frac{3\delta_0}{4K_0}\big)^2$ be a geometrical critical value.
Then, there exists $\varepsilon_*=\varepsilon_*(c) > 0$ such that, for
all $(\mathcal D,h)\in\Hcal$ with
\[
\Fcal(\Dcal,h)\leq c+\varepsilon_*
\]
there exists a continuous map $\eta:[0,1]\times h(1,\Dcal)\to\mathfrak M$ such that $(\eta\star
h,\mathcal D) \in\Hcal$ and
\[
\Fcal\big(\mathcal D\setminus {{\mathcal A}_{c,\varepsilon_*},\eta\star h} \big)
\leq c-\varepsilon_*.
\]
\end{prop}

Then, to conclude the proof of Theorem \ref{thm:main} by minimax arguments we need just the following topological
results.

\begin{prop}\label{thm:prop9.4} Suppose that is only one orthogonal geodesic chord and let $\varepsilon_*$ given by Proposition \ref{thm:prop9.3}. Then, there exists  $\varepsilon \in ]0,\varepsilon_*]$  such that the set  ${\mathcal A}_{c,\varepsilon}$ given in \eqref{eq:Ac} satisfies the following property: there exist an open subset $\widehat {{\mathcal A}_{c,\varepsilon}}\subset\mathfrak C$ containing
${{\mathcal A}_{c,\varepsilon}}$
and a continuous map $h_{c,\varepsilon}:[0,1]\times \widehat {{\mathcal A}_{c,\varepsilon}} \to\mathfrak C$   such that
\begin{enumerate}

\item\label{itm:prop9.4-1} $h_{c,\varepsilon_*}(0,y)=y$, for all  $y\in \widehat {{\mathcal A}_{c,\varepsilon}}$;
\item\label{itm:prop9.4-4} $h_{c,\varepsilon}(1,\widehat {{\mathcal A}_{c,\varepsilon}})=\{y_0\}$ for some $y_0\in\mathfrak C$.
\end{enumerate}
\end{prop}

\begin{proof}[Proof of Proposition \ref{thm:prop9.4}]
 By the Second Deformation Lemma, we deduce the existence of $\epsilon$ such that $\mathcal A_{c,\varepsilon}$ consists of the disjoint union of a finite number of closed sets $C_i$ consisting of curves $x$ with the same number of intervals $[a,b]\in \mathcal I_x$ such that $x_{[a,b]}$ is $r_*$--close to $\omega_{a,b}$.

On any $C_i$, arguing as in \cite{arma}, thanks to the transversality properties of OGC's, we can construct continuous maps $\alpha(x)$ and $\beta(x)$ having the following properties:
\begin{itemize}
\item $\alpha(x) < \beta(x)$,
\item $\dist_{E}(x(\alpha(x)),\omega(0))<2r_*$ or $\dist_{E}(x(\alpha(x)),\omega(1))<2r_*$,
\item $\dist_{E}(x(\beta(x)),\omega(0))<2r_*$ or $\dist_{E}(x(\beta(x)),\omega(1))<2r_*$,
\item if $[a,b] \in \mathcal I_x$ is such that $b \leq \alpha(x)$ or $a \geq \beta(x)$ then $x_{[a,b]}$ is not close to $\omega_{a,b}$.
\end{itemize}
Then, as in the First Deformation Lemma, since $\omega$ is the unique OGC, we see that we can continuously
retract any $x\vert_{[0,\alpha(x)}]$ and $x\vert_{[\beta(x),1]}$ on $\partial \Omega$. Then moving $x(0)$ along $x$ until we reach $x(\alpha(x))$ and
$x(1)$ along $x$ until we reach $x(\beta(x))$ we obtain the searched homotopy on ${\mathcal A}_{c,\varepsilon}$
Finally
Since $\mathfrak C$ is an ANR (\emph{absolute neighborhood retract},
cf.\ \cite{Palais}),
we can immediately extend the obtained homotopy
to a suitable open set $\widehat{{\mathcal A}_{c,\varepsilon}}$, containing ${{\mathcal A}_{c,\varepsilon}}$
and satisfying the required properties.
\end{proof}

\section{Proof of the main Theorem \ref{thm:main}}

The topological invariant that will be employed in the minimax argument is the relative category $\cat$ defined in Section~\ref{main}; recall from Lemma~\ref{thm:estimatecat} that:
\begin{equation}\label{eq:10.1}
\cat_{\mathfrak C,{\mathfrak C_0}}({\mathfrak C})\ge 2.
\end{equation}

Denote by $\mathfrak D$ the class of closed $\Rcal$--invariant subset of $\mathfrak C$. Define, for
any $i=1,2$,
\begin{equation}\label{eq:10.2}
\Gamma_i=\big\{\Dcal\in\mathfrak D\,:\,\cat_{{\mathfrak C},{\mathfrak C}_0}(\Dcal)\ge i\big\}.
\end{equation}
Set
\begin{equation}\label{eq:10.3}
c_i=\inf_{\stackrel{\Dcal\in\Gamma_i,}{(\mathcal D,h)\in\Hcal}}\Fcal(\mathcal D,h).
\end{equation}

\begin{rem}\label{rem:10.2}
If $\mathrm I_{\mathfrak C}:[0,1]\times\mathfrak C$ denotes the map $\mathrm I_{\mathfrak C}(\tau,x)=x$ for
all $\tau$ and all $x$, the the pair $(\mathfrak C,\mathrm I_{\mathfrak C})\in \Hcal$.
Since $\widetilde{\mathfrak C}\in\Gamma_i$ for any $i$ (see \eqref{eq:10.1}),
we get:
\[
c_i\le\Fcal(\mathfrak C,\mathrm I_{\mathfrak C}) < M_0.
\]
Moreover $\Fcal\ge 0$, therefore $0\le c_i\le M_0$ for any $i$ (recall also the definition of $\mathcal F$ and $M_0$).
\end{rem}

We have the following lemmas involving the real numbers $c_i$.

\begin{lem}\label{thm:lem10.3} The following statements hold:
\begin{enumerate}
\item\label{itm:lem10.3-1} $c_1\ge \frac{1}{2}\left(\tfrac{3\delta_0}{4K_0}\right)^2$;
\item\label{itm:lem10.3-2} $c_1\le c_2$.
\end{enumerate}
\end{lem}

\begin{lem}\label{thm:lem10.4} For all $i=1,2$,  $c_i$ is a geometrically critical value.
\end{lem}

\begin{lem}\label{thm:lem10.5} Assume that there is only one OGC in $\overline\Omega$.
Then,
\begin{equation}\label{eq:10.4}
c_1<c_2.
\end{equation}
\end{lem}

\begin{proof}[Proof of Lemma \ref{thm:lem10.3}]
Let us prove \eqref{itm:lem10.3-1}. Assume by contradiction $c_1<\frac12 \left(\tfrac{3\delta_0}{4K_0}\right)^2$,
and take $\varepsilon>0$ such that $c_1+\varepsilon<\frac12 \left(\tfrac{3\delta_0}{4K_0}\right)^2$. By
\eqref{eq:10.2}--\eqref{eq:10.3} there exists $\Dcal_\varepsilon\in\Gamma_1$, and $(\mathcal D_\varepsilon,h_\varepsilon)\in\Hcal$
such that
\[
\Fcal(\mathcal D_\varepsilon,h_\varepsilon)\le c_1+\varepsilon
<\frac12\left(\tfrac{3\delta_0}{4K_0}\right)^2.
\]
Let $h_0$ be the homotopy sending any curve $x$ on $x(\frac12)$, and take $\eta_\varepsilon$ given by Lemma \ref{lem:topological} with $h$ replaced by $h_\varepsilon$.
 Then:
\[
(h_0\star \eta_\varepsilon \star h_\varepsilon(1,\Dcal_\varepsilon)) \text{ consists of constant curves in } \partial \Omega,
\]
(and $h_0\star \eta_\varepsilon \star h_\varepsilon$ does not move the constant curves in $\mathcal D_{\varepsilon}$). Then there exist
a homotopy $K_\varepsilon:[0,1]\times\Dcal_\varepsilon\to{\mathfrak C}$ such that
$K_\varepsilon(0,\cdot)$ is the identity, $K_\varepsilon(1,\Dcal_\varepsilon)\subset{\mathfrak C}_0$
and
\[
K_\varepsilon(\tau,\Dcal_\varepsilon\cap\widetilde{\mathfrak C}_0)\subset{\mathfrak C}_0,\,\forall\tau\in[0,1].
\]
Then $\cat_{{\mathfrak C},{\mathfrak C}_0}(\Dcal_\varepsilon)=0$, in contradiction with the
definition of $\Gamma_1$.

To prove \eqref{itm:lem10.3-2}, observe that by \eqref{eq:10.3}
for any $\varepsilon>0$ there exists $\Dcal\in\Gamma_2$ and $(\mathcal D,h)\in\Hcal$ such that
\[
\Fcal(\mathcal D,h)\le c_2+\varepsilon.
\]
Since $\Gamma_2\subset\Gamma_1$ by definition of $c_1$ we deduce $c_1\le c_2+\varepsilon$, and
\eqref{itm:lem10.3-2} is proved, since $\varepsilon$ is arbitrary.
\end{proof}

\begin{proof}[Proof of Lemma \ref{thm:lem10.4}]
Assume by contradiction that $c_i$ is not a  geometrically critical value for some $i$.
Take $\varepsilon=\varepsilon(c_i)$ as in Proposition~\ref{thm:firstdeflemma}, and $(\Dcal_\varepsilon,h)\in\Hcal$ such that
\[
\Fcal(\Dcal_\varepsilon,h)\le c_i+\varepsilon.
\]
Now let $\eta$ as in Proposition~\ref{thm:firstdeflemma} and take $h_\varepsilon=\eta\star h$. By the same
Proposition,
\[
\Fcal(\Dcal_\varepsilon,h_\varepsilon)\le c_i-\varepsilon,
\]
in contradiction with \eqref{eq:10.3} because $(\Dcal_\varepsilon,h_\varepsilon)\in\Hcal$.
\end{proof}

\begin{proof}[Proof of Lemma \ref{thm:lem10.5}]
Assume by contradiction that \eqref{eq:10.4} does not hold. Then
\[
c\equiv c_1=c_2.
\]
Take $\varepsilon_*=\varepsilon_*(c)$ as in Proposition \ref{thm:prop9.3},
$\Dcal_2\in\Gamma_2$ and $(\Dcal_2,h)\in\Hcal$, such that
\[
\Fcal(\Dcal_2,h)\le c+\varepsilon_*.
\]

Let
$\mathcal A = \widehat{{\mathcal A}_{c,\varepsilon}}$ be the open set given by Proposition \ref{thm:prop9.4}.
The by definition of $\Gamma_1$,
and simple properties of relative category,
\[
\Dcal_1\equiv\Dcal_2\setminus {\mathcal A}\in\Gamma_1.
\]
Now let $\eta$ as in Proposition \ref{thm:prop9.3}. We have
\[
\mathcal F(\mathcal D_2\setminus {\mathcal A},\eta\star h)\leq c-\varepsilon_*,
\]
in contradiction with the definition of $\Gamma_1$.
\end{proof}

\begin{proof}[Proof of Theorem \ref{thm:main}]
It follows immediately from lemmas \ref{thm:lem10.3}--\ref{thm:lem10.5} and Proposition
\ref{thm:distinct}.
\end{proof}

\appendix
\section{An estimate on the relative category}
\label{sec:appendixrelLScat}
Let $n\geq 1$ be an integer; $\mathbb S^n$ is the $n$-dimensional sphere, and $\Delta^n\subset\mathbb S^n\times\mathbb S^n$ is the diagonal. We want to estimate the relative Ljusternik--Schnirelman category of the pair $(\mathbb S^n\times\mathbb S^n,\Delta^n)$, and to this aim we will prove an estimate on the relative cuplength of the pair.

For a topological space $X$ and an integer $k\ge0$, we will denote by $H^k(X)$ and $\widetilde H^k(X)$ respectively the $k$-th singular cohomology and the $k$-th reduced singular cohomology group of $X$. For a topological pair $(X,Y)$, $H^k(X,Y)$ is the $k$-th relative singular cohomology group of the pair; in particular, $H^k(X,\emptyset)=H^k(X)$. Given $\alpha\in H^p(X,Y)$ and $\beta\in H^q(X,Z)$, $\alpha\cup\beta\in H^{p+q}(X,Y\bigcup Z)$ will denote the cup product of $\alpha$ and $\beta$; recall that $\alpha\cup\beta=(-1)^{pq}\beta\cup\alpha$.

The notion of relative cuplength, here recalled, will be also used.

\begin{defin}\label{def:A.4}
The number $\mathrm{cuplength}(X,Y)$ is the largest positive integer $k$ for which there exists $\alpha_0\in H^{q_0}(X,Y)$
($q_0\ge 0$) and $\alpha_i\in H^{q_i}(X)$, $i=1,\ldots,k$ such that
\[
q_i\ge 1,\quad\forall\; i=1,\ldots,k,
\]
and
\[
\alpha_0\cup\alpha_1\cup\ldots\cup\alpha_k\ne 0 \text{\ in\ } H^{q_0+q_1+\ldots+q_k}(X,Y),
\]
where $\cup$ denotes the cup product.
\end{defin}

As for the absolute Lusternik--Schirelmann category, we have the following estimate of relative category by means of relative cuplenght, cf e.g. \cite{FH,FW}
\begin{prop}\label{thm:propA2}
$\cat_{\mathbb S^n \times \mathbb S^n,\Delta^n}(\mathbb S^n \times \mathbb S^n) \geq
\mathrm{cuplength}(\mathbb S^n\times\mathbb S^n,\Delta^n) + 1$.\qed
\end{prop}

Therefore, to prove that $cat_{\mathbb S^n \times \mathbb S^n,\Delta^n}(\mathbb S^n \times \mathbb S^n) \geq 2$ it will be sufficient to prove the following

\begin{prop}\label{thm:propA3}
For all $n\ge1$, $\mathrm{cuplength}(\mathbb S^n\times\mathbb S^n,\Delta^n)\ge1$.
\end{prop}
\begin{proof}
The statement is equivalent to proving the existence of $p\ge0$, $q\ge1$, $\alpha\in H^p(\mathbb S^n\times\mathbb S^n,\Delta^n)$ and $\beta\in H^q(\mathbb S^n\times\mathbb S^n)$ such that $\alpha\cup\beta\ne0$.
This will follow immediately from the Lemma below.
\end{proof}
\begin{lem}
For $n\ge1$, the group $H^{2n}(\mathbb S^n\times\mathbb S^n, \Delta^n)$ is isomorphic to $\mathds Z$, and the map $H^n(\mathbb S^n\times\mathbb S^n,\Delta^n) \times H^n(\mathbb S^n\times\mathbb S^n)\ni(\alpha,\beta)\mapsto\alpha\cup\beta\in H^{2n}(\mathbb S^n\times\mathbb S^n, \Delta^n)$ is surjective.
\end{lem}
\begin{proof}
It is well known that $H^k(\mathbb S^n)\cong\mathds Z$ for $k=0,n$, and $H^k(\mathbb S^n)=0$ if $k\ne0,n$.
It follows $H^n(\mathds S^n\times\mathds S^n)\cong\bigoplus_{k=0}^nH^k(\mathbb S^n)\otimes H^{n-k}(\mathbb S^n)\cong\mathds Z\oplus\mathds Z$. If $\omega$ is a generator of $H^n(\mathbb S^n)$, then the two generators of
$H^n(\mathbb S^n\times\mathbb S^n)\cong\mathds Z\oplus\mathds Z$ are $\pi_1^*(\omega)$ and $\pi_2^*(\omega)$, where
$\pi_1,\pi_2:\mathbb S^n\times\mathbb S^n\to\mathbb S^n$ are the projections.

For the computation of $H^n(\mathbb S^n\times\mathbb S^n,\Delta^n)$, we use the long exact sequence of the pair $(\mathbb S^n\times\mathbb S^n,\Delta^n)$ in reduced cohomology:
\[\cdots\longrightarrow\widetilde H^{n-1}(\Delta^n)\longrightarrow H^n(\mathbb S^n\times\mathbb S^n,\Delta^n)\longrightarrow\widetilde H^n(\mathbb S^n\times\mathbb S^n)\stackrel{\mathfrak i^*}\longrightarrow\widetilde H^n(\Delta^n)\longrightarrow\cdots\]
Since $\Delta^n$ is homeomorphic to $\mathbb S^n$, then $\widetilde H^{n-1}(\Delta^n)=0$. Thus, the group $H^n(\mathbb S^n\times\mathbb S^n,\Delta^n)$ can be identified with the subgroup of $\widetilde H^n(\mathbb S^n\times\mathbb S^n)$ given by the kernel of the map
$\mathfrak i^*:\widetilde H^n(\mathbb S^n\times\mathbb S^n)\to\widetilde H^n(\Delta^n)$.
This map takes each of the two generators $\pi_i^*(\omega)$, $i=1,2$, to $\omega$ (here we identify $\Delta^n$ with $\mathbb S^n$), so that $H^n(\mathbb S^n\times\mathbb S^n,\Delta^n)$ is the subgroup of $\widetilde H^n(\mathbb S^n\times\mathbb S^n)$ generated by $\pi_1^*(\omega)-\pi_2^*(\omega)$, which is isomorphic to $\mathds Z$.

Finally, let us compute $H^{2n}(\mathbb S^n\times\mathbb S^n,\Delta)$ using again the long exact sequence of the pair $(\mathbb S^n\times\mathbb S^n,\Delta^n)$ in reduced cohomology:
\[\cdots\longrightarrow\widetilde H^{2n-1}(\Delta^n)\longrightarrow H^{2n}(\mathbb S^n\times\mathbb S^n,\Delta^n)\longrightarrow\widetilde H^{2n}(\mathbb S^n\times\mathbb S^n)\stackrel{\mathfrak i^*}\longrightarrow\widetilde H^{2n}(\Delta^n)\longrightarrow\cdots\]
Clearly, $\widetilde H^{2n}(\Delta^n)=0$, and if $n>1$, also $\widetilde H^{2n-1}(\Delta^n)=0$.
When $n=1$, then $\widetilde H^{2n-1}(\Delta^n)=\widetilde H^1(\Delta^1)\cong\mathds Z$, however the
map $\widetilde H^1(\Delta^1)\to\widetilde H^2(\mathbb S^1\times\mathbb S^1)$ is identically zero, because the previous map of the exact sequence $\widetilde H^1(\mathbb S^1\times\mathbb S^1)\to\widetilde H^1(\Delta^1)$ is clearly surjective.\footnote{The map $\widetilde H^1(\mathbb S^1\times\mathbb S^1)\to\widetilde H^1(\Delta^1)$
is induced by the diagonal inclusion of $\mathbb S^1$ into $\mathbb S^1\times\mathbb S^1$. It takes both generators $\pi_1^*(\omega)$ and $\pi_2^*(\omega)$ of $H^1(\mathbb S^1\times\mathbb S^1)$ to the generator  $\omega$ of $H^1(\mathbb S^1)\cong H^1(\Delta^1)$.} In both cases, $n=1$ or $n>1$, we obtain
$H^{2n}(\mathbb S^n\times\mathbb S^n,\Delta^n)\cong\widetilde H^{2n}(\mathbb S^n\times\mathbb S^n)\cong\mathds Z$.
A generator of $\widetilde H^{2n}(\mathbb S^n\times\mathbb S^n)$ is $\pi_1^*(\omega)\cup\pi_2^*(\omega)$.

In conclusion, using the above identifications, the map $H^n(\mathbb S^n\times\mathbb S^n,\Delta^n) \times H^n(\mathbb S^n\times\mathbb S^n)\ni(\alpha,\beta)\mapsto\alpha\cup\beta\in H^{2n}(\mathbb S^n\times\mathbb S^n, \Delta^n)$ reads as the bilinear map $\mathds Z\times(\mathds Z\oplus\mathds Z)\to\mathds Z$ that takes
$\big(1,(1,0)\big)$ to $(-1)^{n+1}$ and $\big(1,(0,1)\big)$ to $1$. This is clearly surjective.
\end{proof}
From Proposition~\ref{thm:propA2} and Proposition~\ref{thm:propA3} we get:
\begin{cor}\label{thm:estrelcategory}
For all $n\ge1$, $\cat_{\mathbb S^n \times \mathbb S^n}(\mathbb S^n \times \mathbb S^n,\Delta^n)\ge2$.\qed
\end{cor}


\end{document}